\documentclass[submission,copyright,creativecommons]{eptcs}
\providecommand{\event}{ACT 2023} 

\usepackage{iftex}

\ifpdf
  \usepackage{underscore}         
  \usepackage[T1]{fontenc}        
\else
  \usepackage{breakurl}           
\fi

\title{Overdrawing Urns using Categories of Signed Probabilities}
\author{Bart Jacobs and Dario Stein
\institute{iHub, Radboud University Nijmegen}
\institute{\today} 
\email{bart.jacobs@ru.nl \qquad dario.stein@ru.nl}
}

\newcommand{\titlerunning}{Overdrawing Urns with Signed Probabilities}
\newcommand{\authorrunning}{B. Jacobs \& D. Stein}

\hypersetup{
  bookmarksnumbered,
  pdftitle    = {\titlerunning},
  pdfauthor   = {\authorrunning},
  pdfsubject  = {EPTCS},               
}

\usepackage{quiver}
\usepackage{mathtools}
\usepackage{pdfrender}
\usepackage{nccmath} 
\usepackage{xspace}

\usepackage{hyperref}
\usepackage{xypic}
\usepackage[all]{xy}
\xyoption{all}
\xyoption{2cell}
\UseTwocells
\xyoption{tips}
\SelectTips{cm}{}  

\newtheorem{theorem}{\textbf{Theorem}}

\newtheorem{proposition}[theorem]{\textbf{Proposition}}
\newtheorem{corollary}[theorem]{\textbf{Corollary}}
\newtheorem{definition}[theorem]{\textbf{Definition}}
\newtheorem{example}[theorem]{Example}

\newenvironment{proof}[1][{\bf Proof}]%
{ \begin{trivlist}%
		\item[\hskip \labelsep {\bfseries #1}]%
	}%
	{ \end{trivlist}%
}

\newcommand*{\fatten}[1][.4pt]{%
	\textpdfrender{
		TextRenderingMode=FillStroke,
		LineWidth={\dimexpr(#1)\relax},
	}%
}

\ifdefined\mathsl\else
\DeclareMathAlphabet{\mathsl}{\encodingdefault}{\rmdefault}{\mddefault}{\sldefault}
\SetMathAlphabet{\mathsl}{bold}{\encodingdefault}{\rmdefault}{\bfdefault}{\sldefault}
\fi

\makeatletter
\newcommand{\mathoverlap}[2]{\mathpalette\mathoverlap@{{#1}{#2}}}
\newcommand{\mathoverlap@}[2]{\mathoverlap@@{#1}#2}
\newcommand{\mathoverlap@@}[3]{\ooalign{$\m@th#1#2$\crcr\hidewidth$\m@th#1#3$\hidewidth}}
\makeatother

\newcommand{\QEDbox}{\textcolor{darkgray}{\ensuremath{\blacktriangleleft}}}
\newcommand{\QED}{\hspace*{\fill}$\QEDbox$}

\newcommand{\setin}[3]{\{#1\in#2\;|\;#3\}}

\newcommand{\bigsetin}[3]{\big\{#1\in#2\;\big|\;#3\big\}}

\newcommand{\idmap}[1][]{\ensuremath{\mathsl{id}_{#1}}}

\newcommand{\klafter}{\mathbin{\mathoverlap{\circ}{\cdot}}}
\newcommand{\NNO}{\mathbb{N}}
\newcommand{\R}{\mathbb{R}}

\newcommand{\nnR}{\mathbb{R}_{\geq 0}}
\newcommand{\shortplus}{\ensuremath{{\kern-2pt}+{\kern-2pt}}}
\newcommand{\shortminus}{\ensuremath{{\kern-1.5pt}-{\kern-1.5pt}}}
\newcommand{\finset}[1]{\ensuremath{\boldsymbol{#1}}}
\newcommand{\supp}{\mathsl{supp}}

\newcommand{\flrn}{\ensuremath{\mathsl{flrn}}}
\newcommand{\sample}{\ensuremath{\mathsl{sam}}}
\newcommand{\drawdelete}{\ensuremath{\mathsl{DD}}}
\newcommand{\one}{\ensuremath{\mathbf{1}}}
\newcommand{\zero}{\ensuremath{\mathbf{0}}}
\newcommand{\binomial}[1][]{\ensuremath{\mathsl{bn}[#1]}}
\newcommand{\binomialvector}[1]{\ensuremath{\frak{b}_{#1}}}
\newcommand{\multinomial}[1][]{\ensuremath{\mathsl{mn}[#1]}}
\newcommand{\dualbasename}{\mathsl{dmn}}
\newcommand{\dualbase}[1][]{\dualbasename[#1]}
\newcommand{\multinomialvector}[1]{\ensuremath{\frak{m}_{#1}}}
\newcommand{\dualbasevector}[1]{\ensuremath{\frak{d}_{#1}}}
\newcommand{\hypergeometric}[1][]{\ensuremath{\mathsl{hg}[#1]}}
\newcommand{\sgnhypergeometricname}{\mathsl{shg}}
\newcommand{\sgnhypergeometric}[1][]{\sgnhypergeometricname[#1]}
\newcommand{\bisgnhypergeometric}[1][]{\ensuremath{\mathsl{bshg}[#1]}}
\newcommand{\polya}[1][]{\ensuremath{\mathsl{pol}[#1]}}
\newcommand{\dirden}{\ensuremath{\mathsl{dir}}}
\newcommand{\dirdst}{\ensuremath{\mathsl{Dir}}}
\newcommand{\dualdirdst}{\ensuremath{\mathsl{DDir}}}
\newcommand{\uniform}[1][]{\ensuremath{\mathsl{uf}_{{\kern-.3ex}#1}\xspace}}
\newcommand{\Dst}{\mathcal{D}}
\newcommand{\fullDst}{\ensuremath{\mathcal{D}_{{\kern-0.4pt}\mathsl{fs}}}}
\newcommand{\Mlt}{\mathcal{M}}
\newcommand{\fullMlt}{\ensuremath{\Mlt_{{\kern-0.2pt}\mathsl{fs}}}}
\newcommand{\Giry}{\mathcal{G}}
\newcommand{\Sgn}{\mathcal{S}}

\newcommand{\HP}[2]{\ensuremath{\mathsl{P}_{#2}(\simplex{#1})}}
\newcommand{\intd}{{\kern.2em}\mathrm{d}{\kern.03em}}
\newcommand{\simplex}[1]{\Delta^{#1}}
\newcommand{\Kl}{\mathcal{K}{\kern-.4ex}\ell}
\newcommand{\tuple}[1]{\langle#1\rangle}
\newcommand{\ket}[1]{\ensuremath{|{\kern.1em}#1{\kern.1em}\rangle}}
\newcommand{\bigket}[1]{\ensuremath{\big|{\kern.1em}#1{\kern.1em}\big\rangle}}
\newcommand{\ketstrut}{\vrule height 10pt depth 5pt width 0pt}
\newcommand{\Bigket}[1]{\ensuremath{\left|\left.\ketstrut{\kern.1em}#1{\kern.005em}\right>\right.}}
\newcommand{\facto}[1]{\ensuremath{#1{\kern-2.5pt}\raisebox{-2.5pt}{\includegraphics[width=0.9em]{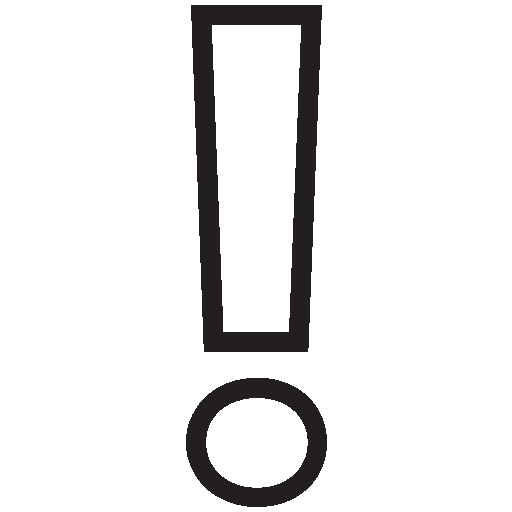}}}}
\newcommand{\coefm}[1]{\ensuremath{\fatten[0.6pt]{(}{\kern1pt}#1{\kern1pt}\fatten[0.6pt]{)}}}
\newcommand{\inprod}[2]{\ensuremath{\fatten[0.6pt]{\langle}{\kern1pt}#1,#2{\kern1pt}\fatten[0.6pt]{\rangle}}}
\newcommand{\bibinom}[2]{\left({\kern-.2em}\binom{#1}{#2}{\kern-.2em}\right)}
\newcommand{\setsize}[1]{|{\kern.1em}#1{\kern.1em}|}
\newcommand{\bigsetsize}[1]{\big|{\kern.1em}#1{\kern.1em}\big|}
\newcommand{\Bigsetsize}[1]{\Big|{\kern.1em}#1{\kern.1em}\Big|}
\newcommand{\push}{\mathrel{\mathchoice%
		{\scalebox{-0.5}[1]{$=\ll$}}
		{\scalebox{-0.5}[1]{$={\kern-1.5ex}\ll$}}
		{\scalebox{-0.5}[1]{${\kern.5ex}\scriptstyle={\kern-0.2ex}\ll{\kern.5ex}$}}
		{\scalebox{-0.5}[1]{$\scriptscriptstyle=\ll$}}}}
\newcommand{\bind}[2]{#1 \push #2}

\usepackage{bm}
\renewcommand{\vec}[1]{\boldsymbol{#1}}

\usepackage{./tikzit}
\usetikzlibrary{circuits.ee.IEC}

\tikzstyle{white dot}=[inner sep=0mm, minimum size=1.5mm, draw=black, shape=circle, text depth=-0.2mm, draw=black, fill=white, tikzit category=nodes]
\tikzstyle{black dot}=[inner sep=0mm, minimum size=1.5mm, draw=black, shape=circle, draw=black, fill=black, tikzit category=nodes]
\tikzstyle{observed}=[inner sep=0mm, minimum size=5mm, draw=black, shape=circle, text depth=-0.2mm, draw=white, tikzit draw=gray, fill=white, tikzit category=dag]
\tikzstyle{latent}=[inner sep=0mm, minimum size=5mm, draw=black, shape=circle, text depth=-0.2mm, draw=black, fill=white, tikzit category=dag]
\tikzstyle{small box}=[shape=rectangle, text height=1.5ex, text depth=0.25ex, yshift=0.5mm, fill=white, draw=black, minimum height=6mm, yshift=-0.5mm, minimum width=6mm, font={\small}, tikzit category=boxes]
\tikzstyle{medium box}=[shape=rectangle, draw=black, fill=white, small box, minimum width=8mm, tikzit category=boxes]
\tikzstyle{semilarge box}=[shape=rectangle, draw=black, fill=white, small box, minimum width=12.5mm, tikzit category=boxes]
\tikzstyle{large box}=[shape=rectangle, draw=black, fill=white, small box, minimum width=15mm, tikzit category=boxes]
\tikzstyle{upground}=[circuit ee IEC, thick, ground, rotate=90, scale=1.5, inner sep=-2mm, tikzit shape=circle, tikzit fill=blue, tikzit category=points]
\tikzstyle{downground}=[circuit ee IEC, thick, ground, rotate=-90, scale=1.5, inner sep=-2mm, tikzit shape=circle, tikzit fill=green, tikzit category=points]
\tikzstyle{point}=[regular polygon, regular polygon sides=3, draw, scale=0.75, inner sep=-0.5pt, minimum width=9mm, fill=white, regular polygon rotate=180, tikzit category=points]
\tikzstyle{copoint}=[regular polygon, regular polygon sides=3, draw, scale=0.75, inner sep=-0.5pt, minimum width=9mm, fill=white, tikzit category=points]
\tikzstyle{uniform}=[point, fill=gray, tikzit shape=circle, scale=0.5]
\tikzstyle{label}=[font={\footnotesize}, text height=1.5ex, text depth=0.25ex, tikzit draw=blue, tikzit fill=white, tikzit category=labels]
\tikzstyle{left label}=[label, anchor=east, xshift=2mm, tikzit draw=green, tikzit fill=white, tikzit category=labels]
\tikzstyle{right label}=[label, anchor=west, xshift=-2mm, tikzit draw=purple, tikzit fill=white, tikzit category=labels]
\tikzstyle{disintegration}=[draw=black, fill={gray!50}, tikzit fill=gray, shape=rectangle, minimum width=1.6cm, minimum height=1.2cm, opacity=0.3]
\tikzstyle{empty diag}=[shape=rectangle, draw=darkgray, dashed, minimum width=8mm, minimum height=8mm, yshift=0.5mm]

\tikzstyle{diredge}=[->, >=latex]
\tikzstyle{dashed edge}=[-, dashed]

\newif\ifexternalizetikz

\usetikzlibrary{shapes,shapes.geometric}
\usetikzlibrary{shapes.gates.ee,shapes.gates.ee.IEC}
\usetikzlibrary{calc}
\usetikzlibrary{positioning}
\usetikzlibrary{cd}
\usetikzlibrary{decorations.pathmorphing}

\ifexternalizetikz
\usepackage{tikzcdx}
\usetikzlibrary{external}
\makeatletter
\newcommand{\tikzextname}[1]{%
	\tikzset{external/figure name={\tikzexternal@realjob-#1-}}}
\makeatother
\fi

\newsavebox\sbpto
\savebox\sbpto{\begin{tikzpicture}[baseline=-2.5pt]
		\filldraw[fill=white,draw=white] circle (1.4pt);
		\filldraw[fill=white,draw=black,line width=0.2pt]circle(2pt);
\end{tikzpicture}}
\newcommand\chanto{\mathrel{\ooalign{$\to$\cr
			\hfil\!$\usebox\sbpto$\hfil\cr}}}            
          
\newsavebox\sbground
\savebox\sbground{\begin{tikzpicture}[circuit ee IEC,yscale=1,xscale=1]
		\draw (0,-2ex) to (0,0) node[ground,rotate=90,xshift=.65ex] {};
\end{tikzpicture}}

\newsavebox\sbunif
\savebox\sbunif{\begin{tikzpicture}[circuit ee IEC,yscale=1,xscale=1]
		\draw (0,0) to (0,2ex) node[ground,rotate=270,xshift=2.5ex] {};
\end{tikzpicture}}

\newcommand{\hide}[1]{}

\begin{document}
\maketitle

\begin{abstract}
A basic experiment in probability theory is drawing without replacement from an urn filled with multiple balls of different colours. Clearly, it is physically impossible to overdraw, that is, to draw more balls from the urn than it contains. This paper demonstrates that overdrawing does make sense mathematically, once we
allow signed distributions with negative probabilities. A new
(conservative) extension of the familiar hypergeometric
(`draw-and-delete') distribution is introduced that allows draws of
arbitrary sizes, including overdraws. The underlying theory makes use
of the dual basis functions of the Bernstein polynomials, which play a prominent role in computer graphics. Negative probabilities are treated systematically in the framework of categorical probability and the central role of datastructures such as multisets and monads is emphasised.
\end{abstract}

\hide{
\begin{abstract}
A basic model in probability theory is an urn filled with multiple
balls of different colours. The probability of drawing a ball of a
certain colour is determined by the proportion of balls of that colour
in the urn. Clearly, it is physically impossible to overdraw, that is,
to draw more balls from the urn than it contains. This paper
demonstrates that overdrawing does make sense mathematically, once we
allow signed distributions with negative probabilities. A new
(conservative) extension of the familiar hypergeometric
(`draw-and-delete') distribution is introduced that allows draws of
arbitrary sizes, including overdraws. The underlying theory makes use
of the dual basis functions with respect to the Bernstein polynomials,
which play a prominent role in computer graphics. It yields an
analogue of a classical De Finetti style result: the infinite chain of
P\'olya (`draw-and-duplicate') distributions, of draws of increasing
sizes, can be obtained via multinomial distributions (of corresponding
sizes) over a single Dirichlet distribution. The paper introduces a
new (signed) dual Dirichlet distribution and shows that the chain of
hypergeometric distributions (of increasing sizes) can be obtained via
multinomials over this dual Dirichlet. Negative probabilities are
treated systematically in the framework of categorical probability and
the central role of datastructures such as multisets and monads is
emphasised.
\end{abstract}
}

\section{Introduction}
For drawing (multiple) coloured balls from a statistical urn, we
distinguish three well-known modes:
\begin{enumerate}
\item \emph{hypergeometric} or \emph{draw-and-delete}, which is
  drawing a ball from the urn without replacement, so that the urn
  shrinks;
	
\item \emph{multinomial} or \emph{draw-and-replace}: drawing with
  replacement, so that the urn remains the same;
	
\item \emph{P\'olya} or \emph{draw-and-duplicate}, which is drawing a
  ball from the urn and replacing it together with an additional ball
  of the same colour, so that the urn grows.
\end{enumerate}

\noindent Multinomial and P\'olya draws may be of arbitrary size, but
hypergeometric draws are limited in size by the number of balls in the
urn. In this paper we lift this limitation and allow hypergeometric
draws of arbitrary size, including `overdraws', containing more balls
than in the urn. Physically this is strange, but, as will show,
mathematically it makes sense once we allow negative probabilities.

Negative probabilities have emerged in quantum physics
(\textit{e.g.}~in double slit experiments) and have been discussed in
the work of famous physicists like Wigner, Dirac, and Feynman (see
\textit{e.g.}~\cite{Feynman87} and the references mentioned
there). There are also `classical' (non-quantum) examples, such as the
one of Piponi (discussed in~\cite{AbramskyB14}) or of
Sz\'ekely~\cite{Szekely05} with two half coins, involving infinitely
many both positive and negative probabilities, whose (convolution) sum
is an ordinary (fair) coin.  Also, negative probabilities have come up
in finance, see \textit{e.g.}~\cite{MeissnerB11}. Despite the lack of
the clear operational meaning that their nonnegative counterparts
have, negative probabilities appear as convenient tools in a variety
of contexts in mathematics and physics,
see~\cite{blass2021negative,tijms2007negative,AbramskyB14,Feynman87}.

We briefly explain the nature of our extension, already using some
notation that will be explained below. The P\'olya distribution can be
expressed as a mixture of multinomial draws; that is, we can break up
such a draw into two stages: first sample a random distribution
$\omega$ from the Dirichlet distribution, and then make independent
(multinomial) draws from $\omega$. The self-reinforcing behaviour of
P\'olya's urn is entirely captured by the latent Dirichlet
distribution. In the Kleisli category of the Giry-monad, with
(Kleisli) morphisms called channels, this corresponds to a
factorisation of the P\'olya channel $\polya[K]$ through the
multinomial channel $\multinomial[K]$, with draws of size $K$.
\begin{equation}
	\begin{tikzcd}
		{\polya[K] \quad = \quad \Mlt[N](X)} & {\Dst(X)} & {\Mlt[K](X)}
		\arrow["\dirdst", from=1-1, to=1-2]
		\arrow["{\multinomial[K]}", from=1-2, to=1-3]
	\end{tikzcd} \label{foo}
\end{equation}

\noindent We write $\Mlt[N](X)$ for the space of multisets (urns) on a
set $X$ of size $N$, and $\Dst(X)$ is the set of finite
distributions. Since such factorisations arise from De Finetti's
famous theorem, we call~\eqref{foo} a De Finetti factorisation for the
P\'olya's distribution (\textit{e.g.}~\cite{JacobsS20}).

The hypergeometric distribution does not admit such a De Finetti
factorisation since an urn containing $N$ balls is exhausted after $N$
draws.  However, there is a way out, if we extend our notion of
probability to allow negative (signed) probabilities. Such models
satisfy the usual axioms of categorical probability, and we can find a
De Finetti factorization of the hypergeometric channel
$\hypergeometric[K]$, for draws of size $K$, of the form:
\begin{equation}
	\begin{tikzcd}
		{\hypergeometric[K] \quad = \quad \Mlt[N](X)} & {\Dst(X)} & {\Mlt[K](X)}
		\arrow["\dualdirdst", from=1-1, to=1-2]
		\arrow["{\multinomial[K]}", from=1-2, to=1-3]
	\end{tikzcd} \label{bar}
\end{equation}

\noindent It uses a signed `Dual Dirichlet' distribution $\dualdirdst$
which we develop in analogy to the Dirichlet distribution occurring in
P\'olya's urn. Existence of the factorisation can be deduced from
earlier
work~\cite{jaynes1982some,diaconis1977finite,kerns2006definetti}
connecting finite versions of De Finetti's theorem to signed
probability. In this case, the factorisation~\eqref{bar} is not
unique. We claim that a canonical choice is given by the Dual
Bernstein polynomials, which have been studied widely in computer
graphics~\cite{lorentz2013bernstein,dong1993dual,zhao1988dual,Juttler98},
but their appearance in a probabilistic context is
novel. Evaluating~\eqref{bar} for overdraws $K \geq N$ defines a
signed extension of the hypergeometric distributions which includes
overdraws, while agreeing with the usual distribution for ordinary
draws $K \leq N$.

Our contributions are:
\begin{enumerate}
\item a principled approach to signed probability (discrete and continuous) using multisets and monads;

\item conceptualizing dual Bernstein polynomials as signed probability densities;

\item defining signed hypergeometric distributions that
  conservatively extend hypergeometric draws while preserving good
  properties;

\item explicating the dual Dirichlet distribution and its conjugate
  prior relationships via string diagrams.
\end{enumerate}

\section{Multisets}\label{MultisetSec}

A multiset, also known as bag, is like a subset except that elements
may occur multiple times. We shall use ket notation $n_{1}\ket{x_1} +
\cdots + n_{k}\ket{x_{k}}$ to describe a multiset with $k$ elements,
where element $x_{i}$, say from a set $X$, occurs $n_{i}\in\NNO$ many
times. Equivalently, such a multiset may be described as a function
$\varphi\colon X \rightarrow\NNO$ with finite support $\supp(\varphi)
\coloneqq \setin{x}{X}{\varphi(x) \neq 0}$. The number
$\varphi(x)\in\NNO$ is the multiplicity of $x\in X$; it says how many
times $x$ occurs in the multiset $\varphi$.  

We shall write $\Mlt(X)$ for the set of multisets with elements from a
set $X$, and $\fullMlt(X)\subseteq \Mlt(X)$ for the subset of
multisets $\varphi$ with full support, that is, with $\supp(\varphi) =
X$. The latter only makes sense when $X$ is a finite set. As canonical
finite sets we write $\finset{n} \coloneqq \{0,1,\ldots, n-1\}$, for
$n\in\NNO$.

The size $\|\varphi\|$ of a multiset $\varphi$ is the total number of
elements, including multiplicities. Thus, $\|\varphi\| \coloneqq
\sum_{x}\varphi(x)$, or, in ket notation, $\|\sum_{i}n_{i}\ket{x_i}\|
= \sum_{i} n_{i}$. We write $\Mlt[K](X) \coloneqq
\bigsetin{\varphi}{\Mlt(X)}{\|\varphi\| = K}$ for the set of multisets
of size $K\in\NNO$. When the set $X$ has $n\geq 1$ elements, the
number of multisets of size $K$ in $\Mlt[K](X)$ is
$\big(\!\binom{n}{K}\!\big) = \binom{n+K-1}{K} =
\frac{(n+K-1)!}{K!\cdot (n-1)!}$. For instance, for a set $X =
\{a,b,c\}$ with three elements there are $\big(\!\binom{3}{3}\!\big) =
\frac{5!}{3!\cdot 2!} = 10$ multisets of size $K=3$, namely:
$3\ket{a}$, $\,3\ket{b}$, $\,3\ket{c}$, $\,2\ket{a}+1\ket{b}$,
$\,2\ket{a}+1\ket{c}$, $\,1\ket{a}+2\ket{b}$, $\,2\ket{b}+1\ket{c}$,
$\,1\ket{a}+2\ket{c}$, $\,1\ket{b}+2\ket{c}$,
$\,1\ket{a}+1\ket{b}+1\ket{c}$. Only the last one has full support.
The factorial $n!$ and binomial coefficients $\binom{n}{m}$ and
$\bibinom{n}{m}$ are extended from numbers to multisets (as
in~\cite{Jacobs21b}).

\begin{definition}
	\label{MltNumberDef}
	Let $\varphi,\psi\in\Mlt(X)$ be two multisets. We define 
	\begin{enumerate}
		\item $\facto{\varphi} \coloneqq \prod_{x} \varphi(x)!$;

		\item $\coefm{\varphi} \coloneqq \frac{\|\varphi\|!}{\facto{\varphi}}$;

		\item $\varphi\leq\psi$ iff $\varphi(x) \leq \psi(x)$ for each $x\in
		X$, and $\varphi \leq_{K} \psi$ iff $\varphi\leq\psi$ and
		$\|\varphi\| = K$;

		\item $(\psi - \varphi)(x) = \psi(x) - \varphi(x)$, when $\varphi \leq
		\psi$;
		
		\item $\binom{\psi}{\varphi} \coloneqq
		\frac{\facto{\psi}}{\facto{\varphi} \cdot \facto{(\psi-\varphi)}} =
		\prod_{x} \binom{\psi(x)}{\varphi(x)}$, when $\varphi\leq\psi$;
		
		\item $\big(\binom{\psi}{\varphi}\big) \coloneqq
		\frac{\facto{(\psi+\varphi-\one}}{\facto{\varphi} \cdot
			\facto{(\psi-\one)}} = \prod_{x}
		\big(\binom{\psi(x)}{\varphi(x)}\big)$ when $\psi$ has full support,
		where $\one = \sum_{x} 1\ket{x}$ is the multiset of singletons.
	\end{enumerate}
\end{definition}

\section{Discrete distributions}\label{DiscDstSec}

A discrete probability distribution $\sum_{i}r_{i}\ket{x_i}$ looks
like a multiset, except that the multiplicities $r_{i}$ are now in the
unit interval $[0,1] \subseteq \R$ and add up to one: $\sum_{i}r_{i} =
1$.  We write $\Dst(X)$ for the set of such distributions with
$x_{i}\in X$.  Alternatively, like for multisets, elements
$\omega\in\Dst(X)$ may be described as functions $\omega \colon X
\rightarrow [0,1]$ with finite support and with $\sum_{x} \omega(x) =
1$. When the set $X$ is finite, we write $\fullDst(X) \subseteq
\Dst(X)$ for the discrete distributions with full support:
$\supp(\omega) = X$. An example is the uniform distribution
$\sum_{x\in X} \frac{1}{n}\ket{x}$, where $n\geq 1$ is the number of
elements of a non-empty set $X$. Concretely, a fair coin is described
by the distribution $\frac{1}{2}\ket{H} + \frac{1}{2}\ket{T}$ for $X =
\{H,T\}$.

Each non-empty multiset $\varphi\in\Mlt(X)$ can be turned into
a distribution via normalisation. We call this frequentist learning,
since it involves learning a distribution by counting, and write
it as:
\begin{equation}
	\label{FlrnEqn}
	\begin{array}{rcccl}
		\flrn(\varphi) 
		& \coloneqq &
		\displaystyle\frac{\varphi}{\|\varphi\|}
		& = &
		\displaystyle\sum_{x\in X} \frac{\varphi(x)}{\|\varphi\|}\,\bigket{x}
		\;\in\; \Dst(X).
	\end{array}
\end{equation}

\noindent This frequentist learning is natural in $X$, but it is
not a map of monads, from (non-empty) multisets to distributions.

The set $\Dst(\finset{n})$ of distributions on $\finset{n} =
\{0,\ldots,n-1\}$ can be identified with the simplex $\simplex{n}
\subseteq \R^{n}$,
where:
\begin{equation}
	\label{SimplexEqn}
	\begin{array}{rcl}
		\simplex{n}
		& \coloneqq &
		\bigsetin{(r_{0}, \ldots, r_{n-1})}{\nnR}{\sum_{i} r_{i} = 1}.
	\end{array}
\end{equation}

\noindent This is commonly called the $n-1$ simplex.

There are three famous `draw' distributions, called multinomial,
hypergeometric and P\'olya. We briefly describe them in the style
of~\cite{Jacobs22a} and refer there for more information. These
distributions are all parameterised by a draw size $K$ and form
distributions on the set $\Mlt[K](X)$ of multisets (as draws). One may
think of $X$ as a set of colors.

\begin{definition}
	\label{DrawDstDef}
	We fix a set $X$ and a number $K\in\NNO$.
	\begin{enumerate}
		\item For a distribution $\omega\in\Dst(X)$, used as
                  abstract urn, the multinomial distribution
                  $\multinomial[K](\omega) \in
                  \Dst\big(\Mlt[K](X)\big)$ is defined as:
		\[ \begin{array}{rcl}
			\multinomial[K](\omega)
			& \coloneqq &
			\displaystyle\sum_{\varphi\in\Mlt[K](X)} \coefm{\varphi} \cdot 
			\prod_{x\in X} \omega(x)^{\varphi(x)} \, \bigket{\varphi}.
		\end{array} \]
		
		\item For an `urn' multiset $\upsilon\in\Mlt(X)$, with size $L
		\coloneqq \|\upsilon\| \geq K$ there is a hypergeometric
		distribution $\hypergeometric[K](\upsilon) \in
		\Dst\big(\Mlt[K](X)\big)$ with:
		\[ \begin{array}{rcl}
			\hypergeometric[K](\upsilon)
			& \coloneqq &
			\displaystyle\sum_{\varphi\leq_{K}\upsilon} \frac{\binom{\upsilon}{\varphi}}{\binom{L}{K}}
			\, \bigket{\varphi}.
		\end{array} \]
		
		\noindent The size restriction $K \leq L$ excludes overdraws.
		
		\item Similarly, for a multiset $\upsilon\in\Mlt[L](X)$ with full
		support, there is the P\'olya distribution 
		$\polya[K](\upsilon) \in \Dst\big(\Mlt[K](X)\big)$ with:
		\[ \begin{array}{rcl}
			\polya[K](\upsilon)
			& \coloneqq &
			\displaystyle\sum_{\varphi\in\Mlt[K](X)} 
			\frac{\big(\binom{\upsilon}{\varphi}\big)}{\big(\binom{L}{K}\big)}
			\, \bigket{\varphi}.
		\end{array} \]
	\end{enumerate}
\end{definition}

Intuitively, in the multinomial case drawn balls are returned to the
urn, so that the urn does not change and can be described abstractly
as a discrete distribution. In the hypergeometric case a drawn ball is
removed from the urn, and in the P\'olya case the drawn ball is
returned together with a new ball of the same colour. Thus in the
hypergeometric case the urns shrinks, whereas in the P\'olya case the
urn grows.

Given two distributions $\omega\in\Dst(X)$ and $\rho\in\Dst(Y)$ we can
form a parallel (tensor) product $\omega\otimes\rho \in \Dst(X\times
Y)$, via $(\omega\otimes\rho)(x,y) = \omega(x)\cdot\rho(y)$.

The mapping $X \mapsto \Dst(X)$ is a monad on the category of sets.
We will not spell out what this means, but we will use the resulting
Kleisli category $\Kl(\Dst)$, whose maps $c\colon X \rightarrow
\Dst(Y)$ will be called channels and written as $c \colon X \chanto
Y$. For instance, the distributions from Definition~\ref{DrawDstDef}
can be described as channels $\multinomial[K] \colon \Dst(X) \chanto
\Mlt[K](X)$, $\hypergeometric[K] \colon \Mlt[L](X) \chanto \Mlt[K](X)$
and $\polya[K] \colon \fullMlt(X) \chanto \Mlt[K](X)$.


For a channel $c\colon X \chanto Y$ and a distribution
$\omega\in\Dst(X)$ on the domain $X$ we can form a distribution
$\bind{c}{\omega}$ on the codomain $Y$ via pushforward (also called
state transformation):
\[ \begin{array}{rcl}
	\big(\bind{c}{\omega}\big)(y)
	& \coloneqq &
	\displaystyle\sum_{x\in X} \omega(x) \cdot c(x)(y).
\end{array} \]

\noindent Given another channel $d\colon Y \chanto Z$ one can form a
composite channel $d \klafter c \colon X \chanto Z$ via $(d \klafter
c)(x) \coloneqq \bind{d}{c(x)}$. Notice that we use a special circle
$\klafter$, with a dot, for composition of channels.

A basic channel is $\drawdelete \colon \Mlt[K\shortplus 1](X) \chanto
\Mlt[K](X)$, where $\drawdelete$ stands for draw-delete. It
probabilistically draws and removes one ball from an urn
$\upsilon\in\Mlt[K\shortplus 1](X)$ with $K\shortplus 1$ balls, via:
\begin{equation}
	\label{DrawDeleteEqn}
	\begin{array}{rcl}
		\drawdelete(\upsilon)
		& \coloneqq &
		\displaystyle\sum_{x\in\supp(\upsilon)} \frac{\upsilon(x)}{\|\upsilon\|}
		\,\bigket{\upsilon-1\ket{x}}.
	\end{array}
\end{equation}

We recall, without proof, the following basic properties of draw
distributions, mostly from~\cite{Jacobs21b}, expressed in terms of
channels.

\begin{proposition}
	\label{DrawDstProp}
	\begin{enumerate}
		
		
		\item \label{DrawDstPropFlrnHg} $\flrn \klafter \hypergeometric[K] =
		\flrn$;
		
		
		\item \label{DrawDstPropFlrnMn} $\flrn \klafter \multinomial[K] =
		\sample$, where $\sample \colon \Dst(X) \chanto X$ is the identity
		map, considered as channel; 
		
		
		\item \label{DrawDstPropHgMn} $\hypergeometric[K] \klafter
		\multinomial[K\shortplus L] = \multinomial[K]$;
		
		
		\item \label{DrawDstPropHgHg} $\hypergeometric[K] \klafter
		\hypergeometric[K\shortplus L] = \hypergeometric[K]$;
		
		\item \label{DrawDstPropHgDD} $\hypergeometric[K] \klafter \drawdelete
		= \hypergeometric[K]$;
		
		\item \label{DrawDstPropDDHg} $\drawdelete \klafter
		\hypergeometric[K\shortplus 1] = \hypergeometric[K]$;
		
		\item \label{DrawDstPropDDMn} $\drawdelete \klafter
		\multinomial[K\shortplus 1] = \multinomial[K]$;
		
		\item \label{DrawDstPropDDPl} $\drawdelete \klafter
		\polya[K\shortplus 1] = \polya[K]$. \QED
	\end{enumerate}
\end{proposition}

The last two items express that multinomial and P\'olya form cones for
the infinite chain of draw-delete channels that appears in a
categorical perspective on De Finetti's theorem, see~\cite{JacobsS20}
and~\cite{jaynes1982some}. This fails in the hypergeometric case since
the draw size $K$ must remain smaller than the size of the urn. 


\section{Continuous distributions}\label{ContDstSec}

In the previous section, we have seen finite discrete probability
distributions over an arbitrary set. There are also continuous
distributions, defined on measurable spaces. Here we need such
distributions only on one particular kind of spaces, namely on
simplices $\simplex{n}$, see~\eqref{SimplexEqn}. The only
distributions that we need are given by (polynomial) functions $f
\colon \simplex{n} \rightarrow \R$ with $\int_{\simplex{n}} f = 1$.
Such an $f$ is called a (probability) density function. It gives rise
to probability measure $\Phi$ that sends a measurable subset $M\subseteq
\simplex{n}$ to the probability $\int_{M} f \in [0,1]$. Such measures
are elements of the set $\Giry\big(\simplex{n}\big)$, where $\Giry$ is
the Giry monad, see \textit{e.g.}~\cite{Panangaden09,Jacobs18c} for
further information.  At first we require that such density functions
are nonnegative, so $f \geq 0$, but later we drop this requirement,
for so-called signed distributions (see the next section).

\begin{definition}
	\label{DirichletDef}
	Let $\upsilon\in\fullMlt(\finset{n})$ be an urn with full support (for
	$n\geq 1$). 
	\begin{enumerate}
		\item \label{DirichletDefDens} It gives rise to the
                  Dirichlet density $\dirden(\upsilon) \colon
                  \simplex{n} \rightarrow \nnR$ given on
                  $\vec{r}\in\simplex{n}$ by:
		\[ \begin{array}{rclcrcl}
			\dirden(\upsilon)(\vec{r})
			& \coloneqq &
			\displaystyle\frac{(\|\upsilon\|-1)!}{\facto{(\upsilon-\one_{n})}} \cdot
			\prod_{i\in\finset{n}} \, r_{i}^{\upsilon(i)-1}
                        & \qquad\mbox{where}\qquad &
                        \one_{n}
                        & \coloneqq &
                        \displaystyle\sum_{0\leq i< n} 1\ket{i} 
                        \,\in\, \fullMlt(\finset{n}).
		\end{array} \]
		
		\noindent We shall drop the index $n$ from $\one_{n}$
                when it is clear from the context.
		
		\item The associated probability measure $\dirdst(\upsilon)$ is
		defined on measurable subsets $M \subseteq \simplex{n}$ as:
		\[ \begin{array}{rcl}
			\dirdst(\upsilon)(M)
			& \coloneqq &
			\displaystyle\int_{\vec{r}\in M} \dirden(\upsilon)(\vec{r}) \intd \vec{r}.
		\end{array} \]
	\end{enumerate}
\end{definition}

The function $\dirden(\upsilon)$ in~\eqref{DirichletDefDens} is a
proper probability density because of the following standard equation
that explains the form of the Dirichlet normalisation constant,
for $\upsilon\in\fullMlt(\finset{n})$.
\begin{equation}
	\label{DirichletNormalisationEqn}
	\begin{array}{rcl}
		\displaystyle\int_{\vec{r}\in\simplex{n}}\,
		\prod_{i\in\finset{n}} \, r_{i}^{\upsilon(i)-1} \intd \vec{r}
		& = &
		\displaystyle\frac{\facto{(\upsilon-\one)}}{(\|\upsilon\|-1)!}.
	\end{array}
\end{equation}

\noindent We use Dirichlet for urns $\upsilon$ with positive natural
numbers as multiplicities. This can be generalised to urns with
positive real numbers as multiplicities --- using the Gamma function
instead of factorials --- but that is not needed in the current
setting.

In the sequel we shall use the bind notation $\bind{c}{\Phi}$ also for
continuous measures, but in a very restricted form, namely for
measures $\Phi$ on $\simplex{n}$ given by a probability density
function $f$ and for channels $c \colon \Dst(\finset{n}) \chanto
\Mlt[K](\finset{n})$. Categorically, this bind is the Kleisli
extension for the Giry monad $\Giry$, see
\textit{e.g.}~\cite{Panangaden09,Fritz20} for details. We will not
elaborate this background and will simply use the relevant equation,
which is of the following form, for $\varphi\in\Mlt[K](\finset{n})$,
\begin{equation}
	\label{ContinuousBindEqn}
	\hspace*{-0.7em}\begin{array}{rcl}
		\bind{c}{\Phi}
		& \coloneqq &
		\displaystyle \!\sum_{\varphi\in\Mlt[K](\finset{n})} \!\!
		\left(\int_{\vec{r}\in\simplex{n}} \!
		c(\vec{r})(\varphi) \cdot f(\vec{r})
		\intd \vec{r}\right)\bigket{\varphi},
        \qquad \mbox{where we identify $\Dst(\finset{n})$ and $\Delta^{n}$.}
	\end{array}
\end{equation}

\begin{figure*}[th]
  \vspace*{-0.5em}
  \[ \hspace*{-2.5em}\hbox{\begin{tikzpicture}
	\begin{pgfonlayer}{nodelayer}
		\node [style=black dot] (0) at (8.25, -1.25) {};
		\node [style=none] (1) at (9.5, 0) {};
		\node [style=none] (2) at (9.5, 3) {};
		\node [style=none] (3) at (5.25, -0.5) {$=$};
		\node [style=small box] (4) at (8.25, -2.25) {$\polya[K]$};
		\node [style=small box] (5) at (2.25, -2.25) {$\dirdst$};
		\node [style=black dot] (6) at (2.25, -1.25) {};
		\node [style=small box] (7) at (3.5, 0.5) {$\multinomial[K]$};
		\node [style=none] (8) at (1, 0.5) {};
		\node [style=none] (9) at (1, 2.25) {};
		\node [style=none] (10) at (3.5, 2.25) {};
		\node [style=black dot] (12) at (8.25, -3.75) {};
		\node [style=none] (14) at (6.75, -0.25) {};
		\node [style=none] (15) at (7.25, -0.25) {};
		\node [style=none] (16) at (7, 3) {};
		\node [style=small box] (17) at (7, 2) {$\dirdst$};
		\node [style=small box] (18) at (7, 0.25) {$+$};
		\node [style=none] (19) at (0, -0.5) {$=$};
		\node [style=small box] (20) at (-3, -2.25) {$\uniform[\Dst(\finset{n})]$};
		\node [style=black dot] (21) at (-3, -1.25) {};
		\node [style=small box] (22) at (-1.75, 0.5) {$\multinomial[K]$};
		\node [style=none] (23) at (-4.25, 0.5) {};
		\node [style=none] (24) at (-4.25, 2.25) {};
		\node [style=none] (25) at (-1.75, 2.25) {};
		\node [style=small box] (26) at (2.25, -4) {\one};
		\node [style=small box] (27) at (8.25, -5) {\one};
		\node [style=black dot] (28) at (13.5, -1.25) {};
		\node [style=none] (29) at (14.75, 0) {};
		\node [style=none] (30) at (14.75, 3) {};
		\node [style=none] (31) at (10.5, -0.5) {$=$};
		\node [style=small box] (32) at (13.5, -2.25) {$\polya[K]$};
		\node [style=none] (34) at (12, -0.25) {};
		\node [style=none] (35) at (12.5, -0.25) {};
		\node [style=none] (36) at (12.25, 3) {};
		\node [style=small box] (37) at (12.25, 2) {$\dirdst$};
		\node [style=small box] (38) at (12.25, 0.25) {$+$};
		\node [style=small box] (39) at (13.5, -4) {\one};
		\node [style=small box] (40) at (12, -4) {\one};
		\node [style=black dot] (41) at (18.75, -1.5) {};
		\node [style=none] (42) at (20, -0.25) {};
		\node [style=none] (43) at (20, 2.75) {};
		\node [style=none] (44) at (15.75, -0.5) {$=$};
		\node [style=none] (46) at (17.25, -0.5) {};
		\node [style=none] (47) at (17.75, -0.5) {};
		\node [style=none] (48) at (17.5, 2.75) {};
		\node [style=small box] (49) at (17.5, 1.75) {$\dirdst$};
		\node [style=small box] (50) at (17.5, 0) {$+$};
		\node [style=small box] (51) at (18.75, -3.75) {\uniform[{\Mlt[K](\finset{n})}]};
		\node [style=small box] (52) at (17.25, -1.75) {\one};
		\node [style=black dot] (53) at (-9, -1.25) {};
		\node [style=none] (54) at (-7.75, 0) {};
		\node [style=none] (55) at (-7.75, 3) {};
		\node [style=none] (56) at (-12, -0.5) {$=$};
		\node [style=small box] (57) at (-9, -2.25) {$\polya[K]$};
		\node [style=small box] (58) at (-15, -2.25) {$\dirdst$};
		\node [style=black dot] (59) at (-15, -1.25) {};
		\node [style=small box] (60) at (-13.75, 0.5) {$\multinomial[K]$};
		\node [style=none] (61) at (-16.25, 0.5) {};
		\node [style=none] (62) at (-16.25, 2.25) {};
		\node [style=none] (63) at (-13.75, 2.25) {};
		\node [style=none] (64) at (-15, -3.75) {};
		\node [style=black dot] (65) at (-9, -3.75) {};
		\node [style=none] (66) at (-9, -4.25) {};
		\node [style=none] (67) at (-10.5, -0.25) {};
		\node [style=none] (68) at (-10, -0.25) {};
		\node [style=none] (69) at (-10.25, 3) {};
		\node [style=small box] (70) at (-10.25, 2) {$\dirdst$};
		\node [style=small box] (71) at (-10.25, 0.25) {$+$};
	\end{pgfonlayer}
	\begin{pgfonlayer}{edgelayer}
		\draw [in=-90, out=30, looseness=1.25] (0) to (1.center);
		\draw (1.center) to (2.center);
		\draw (4) to (0);
		\draw (8.center) to (9.center);
		\draw (7) to (10.center);
		\draw [in=-90, out=30, looseness=1.25] (6) to (7);
		\draw [in=-90, out=150, looseness=1.25] (6) to (8.center);
		\draw (6) to (5);
		\draw (4) to (12);
		\draw [in=150, out=-90, looseness=1.50] (15.center) to (0);
		\draw [in=165, out=-90, looseness=1.25] (14.center) to (12);
		\draw (17) to (16.center);
		\draw (18) to (17);
		\draw (23.center) to (24.center);
		\draw (22) to (25.center);
		\draw [in=-90, out=30, looseness=1.25] (21) to (22);
		\draw [in=-90, out=150, looseness=1.25] (21) to (23.center);
		\draw (21) to (20);
		\draw (5) to (26);
		\draw (12) to (27);
		\draw [in=-90, out=30, looseness=1.25] (28) to (29.center);
		\draw (29.center) to (30.center);
		\draw (32) to (28);
		\draw [in=150, out=-90, looseness=1.50] (35.center) to (28);
		\draw (37) to (36.center);
		\draw (38) to (37);
		\draw (32) to (39);
		\draw (40) to (34.center);
		\draw [in=-90, out=30, looseness=1.25] (41) to (42.center);
		\draw (42.center) to (43.center);
		\draw [in=150, out=-90, looseness=1.50] (47.center) to (41);
		\draw (49) to (48.center);
		\draw (50) to (49);
		\draw (52) to (46.center);
		\draw (41) to (51);
		\draw [in=-90, out=30, looseness=1.25] (53) to (54.center);
		\draw (54.center) to (55.center);
		\draw (57) to (53);
		\draw (61.center) to (62.center);
		\draw (60) to (63.center);
		\draw [in=-90, out=30, looseness=1.25] (59) to (60);
		\draw [in=-90, out=150, looseness=1.25] (59) to (61.center);
		\draw (59) to (58);
		\draw (58) to (64.center);
		\draw (57) to (65);
		\draw (65) to (66.center);
		\draw [in=150, out=-90, looseness=1.50] (68.center) to (53);
		\draw [in=165, out=-90, looseness=1.25] (67.center) to (65);
		\draw (70) to (69.center);
		\draw (71) to (70);
	\end{pgfonlayer}
\end{tikzpicture}} \]
  \vspace*{-1.5em}
  \caption{On the left the equation expressing that multinomial is a
    sufficient statistic for Dirichlet; on the right the string
    diagrammatic proof that the dagger of Dirichlet is multinomial,
    with the uniform distribution as prior, see
    Theorem~\ref{DirichletThm} for details. The boxed $\one$ is the
    point/singleton distribution $1\ket{\one}$. Such point
    distributions commute with copiers.}
	\label{DirichletFig}
\end{figure*}
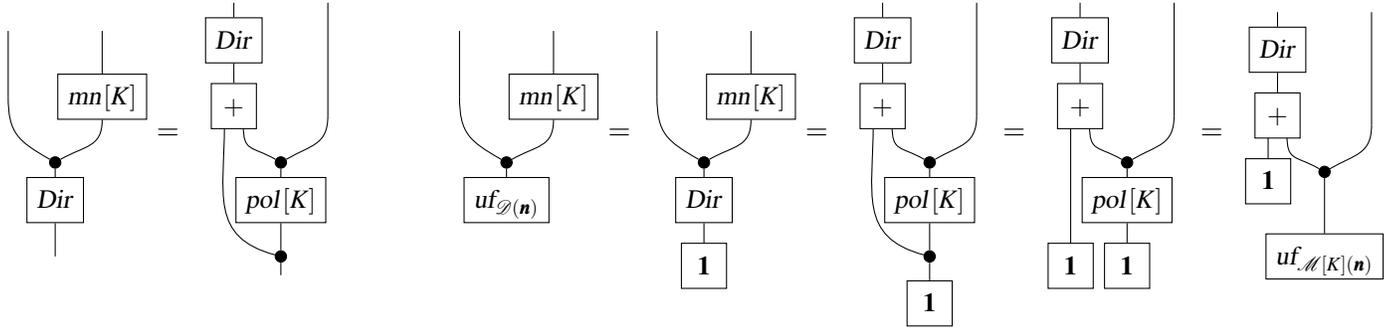

The next result summarises the close relationship between multinomial,
P\'olya, and Dirichlet distributions.  The first two points are
well-known, but the third one probably a bit less --- although it
follows easily from the conjugate prior situation (see
also~\cite{Jacobs20a}). We use string diagrammatic notation, with
flows from bottom to top, since it best displays what is going on,
see~\cite{Fritz20} for details. Proofs are in
Appendix~\ref{ProofAppendix}.

\begin{theorem}
	\label{DirichletThm}
	Let $\upsilon\in\fullMlt(\finset{n})$ be a multiset / urn with full
	support, of size $L \coloneqq \|\upsilon\|$, where $n\geq 1$, and let
	$K$ be an arbitrary number.
	\begin{enumerate}
		\item \label{DirichletThmPol} Multinomial over
                  Dirichlet is P\'olya:
                  $\bind{\multinomial[K]}{\dirdst(\upsilon)} =
                  \polya[K](\upsilon)$.
		
		\item \label{DirichletThmConj} Multinomial is conjugate prior of
		Dirichlet: updating $\dirdst(\upsilon)$ with the predicate /
		likelihood $\multinomial[K](-)(\varphi)$ is $\dirdst\big(\upsilon +
		\varphi\big)$. This is expressed diagrammatically on the left in
		Figure~\ref{DirichletFig}.
		
		\item \label{DirichletThmDag} Multinomial is the
                  dagger of Dirichlet w.r.t.\ the uniform distribution
                  $\uniform[\Dst(\finset{n})]$ on
                  $\Dst(\finset{n})$. In the language / notation
                  of~\cite{ClercDDG17,ChoJ19} this is expressed as:
                  $\multinomial[K] =
                  \dirdst\big(\one\shortplus-\big)^{\dag}_{\uniform[\Dst(\finset{n})]}$.
		
		\item \label{DirichletThmFlrn} When we slightly massage the sample
		channel from Proposition~\ref{DrawDstProp}~\eqref{DrawDstPropFlrnMn}
		to $\sample \colon \simplex{n} \chanto \finset{n}$ given by
		$\sample(\vec{r}) = \sum_{i\in\finset{n}} r_{i}\ket{i}$, then: $\bind{\sample}{\dirdst(\upsilon)} = \flrn(\upsilon)$. \QED
	\end{enumerate}
\end{theorem}

The first item of Theorem~\ref{DirichletThm} tells that P\'olya is
multinomial over Dirichlet. This is an important starting point for
this paper, since we asked ourselves the question whether there is
also a distribution, like Dirichlet, such that multinomial over it is
hypergeometric. We shall see below that the so-called `signed'
Dirichlet distributions achieve this. But first we need to set the
scene for these signed distributions.

\section{Signed distributions}\label{SignDstSec}

We now introduce signed distribution, both in the discrete case and in
the continuous case. As before, we only need continuous distributions
on simplices.

\begin{definition}
	\label{DiscSignedDstDef}
	\begin{enumerate}
		\item A signed discrete probability distribution on a set $X$ is a
		function $\sigma \colon X \rightarrow \R$ with finite support
		$\supp(\sigma) \coloneqq \setin{x}{X}{\sigma(x) \neq 0}$ and with
		$\sum_{x\in X} \sigma(x) = 1$. We may equivalently write such a
		signed discrete distribution in ket notation as a finite formal sum
		$\sum_{i} r_{i}\ket{x_i}$ where $r_{i}\in\R$ satisfy $\sum_{i}r_{i}
		= 1$. We shall write $\Sgn(X)$ for the set of signed discrete probability
		distributions on $X$.
		\item A signed continuous probability distribution, on a simplex
		$\simplex{n}$, is given by a signed density function $f\colon
		\simplex{n} \rightarrow \R$ with $\int_{\vec{r}\in\simplex{n}}
		f(\vec{r}) \intd\vec{r} = 1$.
	\end{enumerate}
\end{definition}

An example of a signed discrete distribution is $\frac{1}{2}\ket{a} -
\frac{1}{4}\ket{b} + \frac{3}{4}\ket{c}$. We do not offer an
operational explanation for what such negative probabilities mean but
treat signed distributions as mathematical objects of their own. It is
not hard to see that signed discrete distributions $\Sgn$ form a monad
on the category of sets and functions. It is affine, in the sense that
$\Sgn(\finset{1}) \cong \finset{1}$, but it differs from $\Dst$ for
instance because it is not \emph{strongly} affine, as defined
in~\cite{Jacobs16a}.

\section{Dual bases}\label{DualBasisSec}


The probability mass function of the multinomial distribution is of a
particularly tractable form, namely a polynomial function $\simplex{n}
\to \R$, on the simplex $\simplex{n}$.


\begin{definition}
	For $\varphi \in \Mlt[K](\finset n)$, we define the \emph{multinomial}
	$\multinomialvector{\varphi}$ as
	\[ \begin{array}{rccclcrcl}
		\multinomialvector{\varphi}(\vec x)
		& \coloneqq &
		\coefm{\varphi} \cdot \vec x^\varphi
		& = &
		\displaystyle\coefm{\varphi} \cdot 
                \prod_{i \in \finset n}\, x_i^{\varphi(i)}
& \qquad\mbox{with `monomial'}\qquad &
	\vec x^\varphi
	& \coloneqq &
	\displaystyle\prod_{i \in \finset n}\, x_i^{\varphi(i)}.
	\end{array} \]
\end{definition}

For every probability vector $\vec{r}\in\simplex{n}$, one has
$\multinomialvector{\varphi}(\vec{r}) =
\multinomial[K](\vec{r})(\varphi)$, via the identification
$\simplex{n} \cong \Dst(\finset{n})$.

	
	\begin{definition}
		\label{PolFunSpaceDef}
		For numbers $n,K$ we write $\HP n K$ for the real
                vector space of polynomial functions $\simplex n \to
                \R$ of degree $K$. The multinomials
                $\multinomialvector{\varphi}$ for
                $\varphi\in\Mlt[K](\finset{n})$ form a basis of this
                space, and so we have as dimension $\dim\big(\HP n
                K\big) = \bibinom n K$.  This vector space $\HP n K$
                is a Hilbert space via an inner product defined on
                $f,g\colon \simplex{n} \rightarrow \R$ as:
		\begin{equation}
			\label{InnerProdEqn}
			\begin{array}{rcl}
				\inprod{f}{g}
				& \coloneqq &
				\displaystyle\int_{\vec{r}\in\simplex{n}} f(\vec{r})\cdot g(\vec{r}) \intd \vec{r}.
			\end{array}
		\end{equation}
	\end{definition}

	The dual of a basis $(b_{i})$ of a space $V$ is generally
        understood as a basis of the dual space $V^*$. In a Hilbert
        space such a dual basis can be described as the elements
        $(d_{i})$ of the space itself which are uniquely determined by
        the relationship $\inprod{b_{i}}{d_{j}} = \delta_{ij}$, so
        that $\inprod{b_{i}}{d_{i}} = 1$ and $\inprod{b_{i}}{d_{j}} =
        0$ for $i\neq j$.
	
	\begin{definition}
		\label{MHPDualDef}
		The \emph{dual multinomials} $(\dualbasevector{\varphi})$ are defined
		as the dual basis of $\HP n K$ to the multinomials
		$(\multinomialvector{\varphi})$, and are as such uniquely
		characterised by the property $\varphi,\psi\in\Mlt[K](\finset{n})$,
		\begin{equation}
			\label{MHPDualInprodEqn}
			\begin{array}{rcccl}
				\inprod{\multinomialvector{\varphi}}{\dualbasevector{\psi}}
				& = &
				\delta_{\varphi,\psi}
				& = &
				\begin{cases}
					1 & \mbox{if }\varphi=\psi
					\\
					0 & \mbox{if }\varphi\neq\psi,
				\end{cases}
			\end{array}
		\end{equation}
	\end{definition}
	
	What do we know about this dual basis? Of course we can
        express the dual basis vectors $\dualbasevector{\psi}$ in
        terms of the original basis, say via scalars $c_{\chi,\psi}$
        satisfying, for each $\psi\in\Mlt[K](\finset{n})$,
	\begin{equation}
		\label{DualBasisEqn}
		\begin{array}{rcccl}
			\dualbasevector{\psi}
			& = &
			\displaystyle\sum_{\chi\in\Mlt[K](\finset{n})} 
			c_{\chi,\psi} \cdot \vec x^{\chi}
			& = &
			\displaystyle\sum_{\chi\in\Mlt[K](\finset{n})} 
			\frac{c_{\chi,\psi}}{\coefm{\chi}} \cdot \multinomialvector{\chi}.
		\end{array}
	\end{equation}
	
	\noindent By exploiting the equations~\eqref{MHPDualInprodEqn}
        and using the linearity of the inner product in each of its
        arguments (\textit{i.e.}~bilinearity), we obtain the equation
	\begin{equation}
		\label{DeltaSolutionEqn}
		\begin{array}{rcccl}
			\delta_{\varphi,\psi}
			& = &
			\inprod{\multinomialvector{\varphi}}{\dualbasevector{\psi}}
			& = &
			\displaystyle\sum_{\chi\in\Mlt[K](\finset{n})} 
			c_{\chi,\psi} \cdot 
			\frac{\inprod{\multinomialvector{\varphi}}{\multinomialvector{\chi}}}
			{\coefm{\chi}}. 
		\end{array}
	\end{equation}
	
	\noindent We note that:
	\[ \begin{array}{rcl}
		\inprod{\multinomialvector{\varphi}}{\multinomialvector{\chi}}
		\hspace*{\arraycolsep}\smash{\stackrel{\eqref{InnerProdEqn}}{=}}\hspace*{\arraycolsep}
		\coefm{\varphi}\cdot\coefm{\chi}\cdot \displaystyle\int_{\vec{r}\in\simplex{n}} \!\!
		\left(\prod_{i\in\finset{n}} r_{i}^{\varphi(i)}\right) \!\cdot\!
		\left(\prod_{i\in\finset{n}} r_{i}^{\chi(i)}\right)\!\! \intd\vec{r}
		 & = &
		\coefm{\varphi}\cdot\coefm{\chi}\cdot \displaystyle\int_{\vec{r}\in\simplex{n}} \!
		\prod_{i\in\finset{n}} r_{i}^{(\varphi+\chi)(i)} \intd\vec{r}
		\\
		& \smash{\stackrel{\eqref{DirichletNormalisationEqn}}{=}} &
		\coefm{\varphi}\cdot\coefm{\chi}\cdot \displaystyle
		\frac{\facto{(\varphi+\chi)}}{(2K\shortplus n\shortminus 1)!}
	\end{array} \]
	
	\noindent There are three square matrices at hand, of size
        $\bibinom{n}{K}\times\bibinom{n}{K}$, with multisets as
        indices, namely:
	\[ \begin{array}{rclcrclcrcl}
		C 
		& = &
		\Big(c_{\varphi,\psi}\Big)_{\!\varphi,\psi\in\Mlt[K](\finset{n})}
		& \qquad &
		\mathsl{FS}
		& = &
		\Big(\facto{(\varphi+\psi)}\!\Big)_{\!\varphi,\psi\in\Mlt[K](\finset{n})}
		& \qquad &
		\mathsl{D}
		& = &
		\Big(\coefm{\varphi}\cdot\delta_{\varphi,\psi}\Big)_{\!\varphi,\psi\in\Mlt[K](\finset{n})}.
	\end{array} \]
	
	
	\noindent This $C$ is the matrix of scalars that we are
        looking for, $\mathsl{FS}$ contains the factorials-of-sums of
        multisets, and $\mathsl{D}$ is a diagonal matrix with multiset
        coefficients.  Equation~\eqref{DeltaSolutionEqn} can now be
        written as:
	\[ \begin{array}{rcccl}
		(2K\shortplus n\shortminus 1)!\cdot\delta_{\varphi,\psi}
		& = &
		\displaystyle\sum_{\chi\in\Mlt[K](\finset{n})} 
		\coefm{\varphi}\cdot \mathsl{FS}_{\varphi,\chi} \cdot C_{\chi,\psi}
		& = &
		\Big(\mathsl{D} \cdot \mathsl{FS} \cdot C\Big)_{\varphi,\psi}.
	\end{array} \]
	
	\noindent We then get $(2K+n-1)!\cdot \mathsl{FS}^{-1} \cdot
	\mathsl{D}^{-1} = C$, so that the coefficients that we seek are
	obtained as:
	\begin{equation}
		\label{DualBasisCoefficientEqn}
		\begin{array}{rcl}
			c_{\varphi,\psi}
			& = &
			\displaystyle \frac{(2K\shortplus n\shortminus 1)!}{\coefm{\psi}}\cdot 
			\Big(\mathsl{FS}^{-1}\Big)_{\!\varphi,\psi}.
		\end{array}
	\end{equation}
	
	\noindent These matrix inverses $\mathsl{FS}^{-1}$ exist since
	$\mathsl{FS}$ is a symmetric positive definite matrix. We give an extended example calculation in the appendix (Example~\ref{BaseCoefficientEx}). \\


	The following is a crucial property of dual bases, as
	introduced in Definition~\ref{MHPDualDef}.

	\begin{proposition}
		\label{DualBaseSgnDstProp}
		Each dual basis function $\dualbasevector{\psi} \in
                \HP n K$, associated with a multiset
                $\psi\in\Mlt[K](\finset{n})$, is a continuous signed
                probability density on $\simplex{n}$, that is:
		\[ \begin{array}{rcl}
			\displaystyle\int_{\vec{r}\in\simplex{n}} \! \dualbasevector{\psi}(\vec{r})
			\intd\vec{r}
			& = &
			1.
		\end{array} \]
	\end{proposition}

	\begin{proof}
		We use that multinomial distributions $\multinomial[K]$ form a
		probability distribution; this means that the multinomials
		$\multinomialvector{\varphi}$ form a partition of unity,
		\textit{i.e.}~for all $\vec r \in \simplex n$:
		\[ \begin{array}{rcl}
			\displaystyle\sum_{\varphi \in \Mlt[K](\finset n)} \multinomialvector{\varphi}(\vec r)
			& = &
			1.
		\end{array} \]
		
		\noindent Hence we obtain
		\[ \begin{array}[b]{rcl}
			\displaystyle\int_{\vec{r}\in\simplex{n}} \! \dualbasevector{\psi}(\vec{r})
			\intd\vec{r}
			\hspace*{\arraycolsep}=\hspace*{\arraycolsep}
			\displaystyle\int_{\vec{r}\in\simplex{n}} \! 1\cdot\dualbasevector{\psi}(\vec{r})
			\intd\vec{r}
			& = &
			\displaystyle\int_{\vec{r}\in\simplex{n}} \! 
			\left(\sum_{\varphi\in\Mlt[K](\finset{n})} 
			\multinomialvector{\varphi}(\vec{r})\right)
			\cdot\dualbasevector{\psi}(\vec{r}) \rlap{$\intd\vec{r}$}
			\\[+1.4em]
			& = &
			\displaystyle\sum_{\varphi\in\Mlt[K](\finset{n})} \int_{\vec{r}\in\simplex{n}} \! 
			\multinomialvector{\varphi}(\vec{r})
			\cdot\dualbasevector{\psi}(\vec{r}) \intd\vec{r}
			\\[+1.4em]
			& \smash{\stackrel{\eqref{InnerProdEqn}}{=}} &
			\displaystyle\sum_{\varphi\in\Mlt[K](\finset{n})} 
			\inprod{\multinomialvector{\varphi}}{\dualbasevector{\psi}}
			\hspace*{\arraycolsep}\smash{\stackrel{\eqref{MHPDualInprodEqn}}{=}}\hspace*{\arraycolsep}
			\displaystyle\sum_{\varphi\in\Mlt[K](\finset{n})} 
			\delta_{\varphi,\psi}
			\hspace*{\arraycolsep}=\hspace*{\arraycolsep}
			1.
		\end{array} \eqno{\blacktriangleleft} \]
	\end{proof}

	In the beginning of this proof, we use that the pointwise sum of
	the multinomial basis functions $\multinomialvector{\varphi}$, for $\varphi \in
	\Mlt[K](\finset n)$, is equal to the constant-one function $\one \colon
	\Delta^{n} \rightarrow \R$. The sum of the dual basis functions
	$\dualbasevector{\varphi}$ is also constant.

	\begin{proposition}
		\label{DualBaseSumProp}
		Fix numbers $n,K$ and consider the dual basis function
		$\dualbasevector{\varphi} \in \HP{n}{K}$, for $\varphi \in
		\Mlt[K](\finset n)$. Then their pointwise sum is a constant function:
		\[ \begin{array}{rcl}
			\displaystyle\sum_{\varphi \in \Mlt[K](\finset n)} \! \dualbasevector{\varphi}
			& \,=\, &
			\displaystyle\frac{(K + n - 1)!}{K!}.
		\end{array} \]
	\end{proposition}

	\begin{proof}
		Since the vectors $\dualbasevector{\varphi}$ form a
                basis we can express the constant-one function $\one
                \colon \Delta^{n} \rightarrow \R$ with respect to this
                basis, say as: $\one = \sum_{\varphi \in
                  \Mlt[K](\finset n)} a_{\varphi} \cdot
                \dualbasevector{\varphi}$, for certain coefficients
                $a_{\varphi}$. For a fixed multiset
                $\psi\in\Mlt[K](\finset n)$ we compute the constant
                $a_{\psi}$ as follows.
		\[ \begin{array}{rcl}
			a_{\psi}
                        & = &
			\displaystyle\sum_{\varphi \in \Mlt[K](\finset n)} a_{\varphi} \cdot \delta_{\psi,\varphi}
			\hspace*{\arraycolsep}=\hspace*{\arraycolsep}
			\displaystyle\sum_{\varphi \in \Mlt[K](\finset n)} a_{\varphi} \cdot 
			\inprod{\multinomialvector{\psi}}{\dualbasevector{\varphi}}
			\hspace*{\arraycolsep}=\hspace*{\arraycolsep}
			\displaystyle \inprod{\multinomialvector{\psi}}
			{\sum_{\varphi \in \Mlt[K](\finset n)} a_{\varphi} \cdot \dualbasevector{\varphi}}
			\\[+1.2em]
			& = &
			\inprod{\multinomialvector{\psi}}{\one}
			\hspace*{\arraycolsep}=\hspace*{\arraycolsep}
			\displaystyle\int_{\vec{r}\in\simplex{n}} \! 
			\multinomialvector{\psi}(\vec{r}) \intd\vec{r}
			\hspace*{\arraycolsep}=\hspace*{\arraycolsep}
			\displaystyle \coefm{\psi} \cdot \int_{\vec{r}\in\simplex{n}} \! 
			\vec{r}^{\psi}(\vec{r}) \intd\vec{r}
			\hspace*{\arraycolsep}\smash{\stackrel{\eqref{DirichletNormalisationEqn}}{=}}\hspace*{\arraycolsep}
			\displaystyle \frac{K!}{\facto{\psi}} \cdot 
			\frac{\facto{\psi}}{(K \shortplus n \shortminus 1)!}
			\hspace*{\arraycolsep}=\hspace*{\arraycolsep}
			\displaystyle \frac{K!}{(K \shortplus n \shortminus 1)!}.
		\end{array} \]
		
		\noindent Thus, all these constants $a_{\psi}$ are the same. As a
		result:
		\[ \begin{array}{rcccl}
			\one
			& = &
			\displaystyle\sum_{\varphi \in \Mlt[K](\finset n)} 
			\frac{K!}{(K \shortplus n \shortminus 1)!} \cdot \dualbasevector{\varphi}
			& = &
			\displaystyle \frac{K!}{(K \shortplus n \shortminus 1)!} \cdot 
			\sum_{\varphi \in \Mlt[K](\finset n)} \dualbasevector{\varphi}.
		\end{array} \]
		
		\noindent By moving the fraction to the other side we are done. \QED
	\end{proof}

\section{Dual Dirichlet and signed hypergeometric}\label{SgnDirHypgeomSec}
	
In the previous section we have introduced the dual basis vectors
$\dualbasevector{\varphi}$ as duals to the multinomial vectors
$\multinomialvector{\varphi}$ and have seen that each of these
$\dualbasevector{\varphi}$ forms a signed probability density. We can
now start harvesting results, first by defining the associated
continuous probability measure.
	
	\begin{definition}
		\label{DualDirichletDef}
		Let $\upsilon\in\Mlt(\finset{n})$ be an multiset (thought of as an urn).
		\begin{enumerate}  
			\item We write $\dualdirdst(\upsilon)$ for the signed probability measure on
			$\simplex{n}$ given by the density $\dualbasevector{\upsilon}$. We
			call it the \emph{dual Dirichlet} distribution.
			
			\item For each number $K$ we define the \emph{signed hypergeometric}
			channel as the composite with multinomial draws
			\[ \begin{array}{rcl}
				\sgnhypergeometric[K]
				& \coloneqq &
				\multinomial[K] \klafter \dualdirdst \,\colon\, 
				\Mlt(\finset{n}) \rightarrow \Sgn\big(\Mlt[K](\finset{n})\big).
			\end{array} \]
			
			\noindent This means:
			\[ \begin{array}{rcccl}
				\sgnhypergeometric[K](\upsilon)
				& = &
				\bind{\multinomial[K]}{\dualdirdst(\upsilon)}   
				& = &
				\displaystyle\!\sum_{\varphi\in\Mlt[K](\finset{n})} \!\!
				\left(\int_{\vec{r}\in\simplex{n}} \!
				\multinomialvector{\varphi}(\vec{r}) \cdot 
				\dualbasevector{\upsilon}(\vec{r})\intd\vec{r}\right)\,\bigket{\varphi}.
			\end{array} \]
		\end{enumerate}
	\end{definition}

In Example~\ref{SgnHypergeomEx} in Appendix~\ref{ExampleAppendix} a
signed hypergeometric distribution is computed concretely. In the
multivariate case we do not have an explicit formula, as exists in the
bivariate case, see Equation~\eqref{BivariateEqn} in
Appendix~\ref{BivariateAppendix} for details. 
		
The next result shows that when there is no overdrawing, there is no
difference between signed and ordinary hypergeometric
probabilities. This means that the dual Dirichlet distribution
confirms the question that we originally set ourselves: there is a
distribution over which multinomials yield hypergeometric
distributions, in analogy with
Theorem~\ref{DirichletThm}~\eqref{DirichletThmPol}.
	
	\begin{theorem}
		\label{SgnDirichletHypgeomThm}
		Let urn $\upsilon\in\Mlt(\finset{n})$ have size $L =
                \|\upsilon\|$.  Then $\sgnhypergeometric[K](\upsilon)
                = \hypergeometric[K](\upsilon)$, for each $K\leq L$.
	\end{theorem}

	This says that when the size of the draw is at most the size of the
	urn --- so when there are no overdraws --- signed hypergeometric
	coincides with ordinary hypergeometric. In particular, in this case no
	negative probabilities appear in the signed hypergeometric.

	\begin{proof}
		We first note that:
		\[ \begin{array}{rcccccl}
			1
			& \smash{\stackrel{\eqref{MHPDualInprodEqn}}{=}} &
			\inprod{\multinomialvector{\upsilon}}{\dualbasevector{\upsilon}}
			& = &
			\displaystyle\int_{\vec{r}\in\simplex{n}} \! 
			\multinomialvector{\upsilon}(\vec{r}) \cdot \dualbasevector{\upsilon}(\vec{r})
			\intd\vec{r}
			& \smash{\stackrel{\eqref{ContinuousBindEqn}}{=}} &
			\Big(\bind{\multinomial[L]}{\dualdirdst(\upsilon)}\Big)(\upsilon).
		\end{array} \]
		
		\noindent As a result: $\bind{\multinomial[L]}{\dualdirdst(\upsilon)}
		= 1\ket{\upsilon}$. By combining this fact with
		Proposition~\ref{DrawDstProp}~\eqref{DrawDstPropHgMn} we are done:
		\[ \begin{array}[b]{rcl}
			\hypergeometric[K](\upsilon)
			\hspace*{\arraycolsep}=\hspace*{\arraycolsep}
			\bind{\hypergeometric[K]}{1\ket{\upsilon}}
			& = &
			\bind{\hypergeometric[K]}
			{\Big(\bind{\multinomial[L]}{\dualdirdst(\upsilon)}\Big)}
			\\[+0.2em]
			& = &
			\bind{\Big(\hypergeometric[K] \klafter \multinomial[L]\Big)}
			{\dualdirdst(\upsilon)}
			\hspace*{\arraycolsep}=\hspace*{\arraycolsep}
			\bind{\multinomial[K]}{\dualdirdst(\upsilon)}.
		\end{array} \eqno{\blacktriangleleft} \]
	\end{proof}

	\begin{corollary}
		\label{SgnDirichletHypgeomCor}
		Let $\upsilon\in\Mlt[L](\finset{n})$ be a multiset/urn of size $L\geq 1$.
		\begin{enumerate}
			\item \label{SgnDirichletHypgeomCorPow} For each multiset
			$\varphi\leq_{K}\upsilon$, that is, for each
			$\varphi\in\Mlt[K](\finset{n})$ with $\varphi\leq\upsilon$, and thus
			$K\leq L$,
			\[ \begin{array}{rcl}
				\displaystyle\int_{\vec{r}\in\simplex{n}} \!
				\vec{r}^{\,\varphi} \cdot \dualbasevector{\upsilon}(\vec{r})\intd\vec{r}
				& = &
				\displaystyle\frac{\facto{\upsilon}\cdot (L-K)!}
				{\facto{(\upsilon-\varphi)}\cdot L!}.
			\end{array} \]
			
			\item \label{SgnDirichletHypgeomCorVar} In particular, for each
			$i\in\supp(\upsilon) \subseteq \finset{n}$.
			\[ \begin{array}{rcl}
				\displaystyle\int_{\vec{r}\in\simplex{n}} \!
				r_{i} \cdot \dualbasevector{\upsilon}(\vec{r})\intd\vec{r}
				& = &
				\flrn(\upsilon)(i).
			\end{array} \]
		\end{enumerate}
	\end{corollary}

	\begin{proof}
		\begin{enumerate}
			\item Since:
			\[ \hspace*{-3em}\begin{array}{rcl}
				\displaystyle\int_{\vec{r}\in\simplex{n}} \!
				\vec{r}^{\,\varphi} \cdot \dualbasevector{\upsilon}(\vec{r})\intd\vec{r}
			\hspace*{\arraycolsep}=\hspace*{\arraycolsep}
				\displaystyle\frac{1}{\coefm{\varphi}}\cdot \int_{\vec{r}\in\simplex{n}} \!
				\multinomialvector{\varphi}(\vec{r}) \cdot 
				\dualbasevector{\upsilon}(\vec{r})\intd\vec{r}
				& = &
				\displaystyle\frac{1}{\coefm{\varphi}}\cdot 
				\Big(\bind{\multinomial[K]}{\dualdirdst(\upsilon)}\Big)(\varphi)
				\\[+0.5em]
				& = &
				\displaystyle\frac{1}{\coefm{\varphi}}\cdot 
				\hypergeometric[K](\upsilon)(\varphi)
			\hspace*{\arraycolsep}=\hspace*{\arraycolsep}
				\displaystyle \frac{\facto{\varphi}}{K!} \cdot 
				\frac{\binom{\upsilon}{\varphi}}{\binom{L}{K}}
				\hspace*{\arraycolsep}=\hspace*{\arraycolsep}
				\displaystyle\frac{\facto{\upsilon}\cdot (L-K)!}
				{\facto{(\upsilon-\varphi)}\cdot L!}.
			\end{array} \]
			
			\item We apply the previous point with $\varphi = 1\ket{i}$ of
			size $1$. Then:
			\[ \hspace*{-2em}\begin{array}{rcccccccccl}
				\displaystyle\int_{\vec{r}\in\simplex{n}} \!
				r_{i} \cdot \dualbasevector{\upsilon}(\vec{r})\intd\vec{r}
				& = &
				\displaystyle\int_{\vec{r}\in\simplex{n}} \!
				\vec{r}^{\,1\ket{i}} \cdot \dualbasevector{\upsilon}(\vec{r})\intd\vec{r}
				& = &
				\displaystyle\frac{\facto{\upsilon}\cdot (L-1)!}
				{\facto{(\upsilon-1\ket{i})}\cdot L!}
				& = &
				\displaystyle\frac{\upsilon(i)}{L}
				& = &
				\displaystyle\frac{\upsilon(i)}{\|\upsilon\|}
				& = &
				\flrn(\upsilon)(i).
			\end{array} \eqno{\blacktriangleleft} \]
		\end{enumerate}
	\end{proof}

	\begin{proposition}
		\label{SgnDirichletProp}
		\begin{enumerate}
			\item \label{SgnDirichletPropZero} We write $\zero$ for the empty
			multiset and $\one$ for the multiset of singletons, say on
			$\finset{n}=\{0,\ldots,n-1\}$. Then $\dualdirdst(\zero) =
			\dirdst(\one)$ is the uniform measure on $\simplex{n}$.
			
			\item \label{SgnDirichletPropSample} The sample channel $\sample
			\colon \simplex{n} \rightarrow \Dst(\finset{n})$ from
			Theorem~\ref{DirichletThm}~\eqref{DirichletThmFlrn} gives, like for
			ordinary Dirichlet,
			\[ \begin{array}{rcl}
				\bind{\sample}{\dualdirdst(\upsilon)}
				& = &
				\flrn(\upsilon).
			\end{array} \]
		\end{enumerate}
	\end{proposition}

	\begin{proof}
		\begin{enumerate}
			\item For $K=0$ the set $\Mlt[K](\finset{n})$ of multisets of size $0$
			contains the empty multiset $\zero$ as sole element. The associated
			factor sum matrix $\mathsl{FS}$ from~\eqref{DualBasisCoefficientEqn}
			is thus the singleton matrix $(1)$, with $(1)$ as inverse. Hence the
			only coefficient $c_{\zero,\zero}$ of the polynomial
			$\dualbasevector{\zero}$ in~\eqref{DualBasisCoefficientEqn} is $(n-1)!$.
			This makes $\dualbasevector{\zero}$ the constant function $\vec{r} \mapsto
			(n-1)!$, which is the density of the uniform measure $\dirdst(\one)$
			on $\simplex{n}$.
			
			\item Using
			Corollary~\ref{SgnDirichletHypgeomCor}~\eqref{SgnDirichletHypgeomCorVar},
			the reasoning is precisely as in the proof of
			Theorem~\ref{DirichletThm}~\eqref{DirichletThmFlrn}. \QED
		\end{enumerate}
	\end{proof}

Theorem~\ref{SgnDirichletHypgeomThm} says that the signed
hypergeometric distribution is a `conservative' extension of the
ordinary hypergeometric distribution in the sense that these
distributions coincide for draws of size below the size of the urn. We
now illustrate draws of arbitrary size, also bigger than the size of
the urn.  The physical interpretation of such overdraws is
unclear. But mathematically all works well.
	
We continue with some basic properties of signed hypergeometric
distributions, as analogues of (some of the items of)
Proposition~\ref{DrawDstProp}.

	\begin{proposition}
          \label{SgnHypgeomProp}
		\begin{enumerate}
			\item $\sgnhypergeometric[K](\zero)$ is the uniform distribution on
			$\Mlt[K](\finset{n})$;
			
			\item $\flrn \klafter \sgnhypergeometric[K] = \flrn$;
			
			\item $\sgnhypergeometric[K] \klafter \multinomial[L\shortplus K] =
			\multinomial[K]$;
			
			\item $\sgnhypergeometric[K] \klafter \sgnhypergeometric[K\shortplus L] =
			\sgnhypergeometric[K]$;
			
			\item $\drawdelete \klafter \sgnhypergeometric[K\shortplus 1] =
			\sgnhypergeometric[K]$.
		\end{enumerate}
	\end{proposition}

	The last equation shows that the signed hypergeometric form a
        cone for the draw-delete maps and thus fit in a categorial
        approach to `De Finetti', following~\cite{JacobsS20}. The
        equation $\hypergeometric[K] \klafter \drawdelete =
        \hypergeometric[K]$ from
        Proposition~\ref{DrawDstProp}~\eqref{DrawDstPropHgDD} holds
        for ordinary hypergeometric, but its analogue for signed
        hypergeometric fails.

\begin{proof} 
\begin{enumerate}
\item Via
  Proposition~\ref{SgnDirichletProp}~\eqref{SgnDirichletPropZero}:
  $\sgnhypergeometric[K](\zero) =
  \bind{\multinomial[K]}{\dualdirdst(\zero)} =
  \bind{\multinomial[K]}{\dirdst(\one)} = \polya[K](\one)$. The latter
  P\'olya distribution on $\Mlt[K](\finset{n})$ is uniform, see
  Theorem~\ref{DirichletThm}.
			
\item By combining
  Proposition~\ref{DrawDstProp}~\eqref{DrawDstPropFlrnMn} with
  Proposition~\ref{SgnDirichletProp}~\eqref{SgnDirichletPropSample} we
  get: $\flrn \klafter \sgnhypergeometric[K] = \flrn \klafter
  \multinomial[K] \klafter \dualdirdst = \sample \klafter \dualdirdst
  = \flrn$.
			
\item In the composite $\sgnhypergeometric[K] \klafter
  \multinomial[L\shortplus K]$ the signed hypergeometric draws of size
  $K$ are applied to the urns that appear as draws of size
  $L\shortplus K$ coming out of the multinomial
  $\multinomial[L\shortplus K]$. Hence the draw size is less than the
  urn size, so Theorem~\ref{SgnDirichletHypgeomThm} applies, and the
  signed hypergeometric is an ordinary hypergeometric. Thus
  Proposition~\ref{DrawDstProp}~\eqref{DrawDstPropHgMn} gives:
  $\sgnhypergeometric[K] \klafter \multinomial[L\shortplus K] =
  \hypergeometric[K] \klafter \multinomial[L\shortplus K] =
  \multinomial[K]$.
			
\item By the previous point: $\sgnhypergeometric[K] \klafter
  \sgnhypergeometric[L\shortplus K] = \sgnhypergeometric[K] \klafter
  \multinomial[L\shortplus K] \klafter \dualdirdst = \multinomial[K]
  \klafter \dualdirdst = \sgnhypergeometric[K]$.
			
\item Via Proposition~\ref{DrawDstProp}~\eqref{DrawDstPropDDMn}:
  $\drawdelete \klafter \sgnhypergeometric[K\shortplus 1] =
  \drawdelete \klafter \multinomial[K\shortplus 1] \klafter
  \dualdirdst = \multinomial[K] \klafter \dualdirdst =
  \sgnhypergeometric[K]$. \QED
\end{enumerate}
\end{proof}

	\section{Signed hypergeometric channels as Bayesian inversion}\label{DaggerSec}
	
	This section shows that the signed hypergeometric channel
        $\sgnhypergeometric[K] \colon \Mlt(\finset{n}) \rightarrow
        \Sgn\big(\Mlt[K](\finset{n})\big)$ can be obtained as dagger,
        that is as Bayesian inversion,
        see~\cite{ClercDDG17,ChoJ19}. For this we need the following
        new signed distribution, that builds on the result from
        Proposition~\ref{DualBaseSumProp} that the sum of dual basis
        functions is constant.
	
	\begin{definition}
		\label{DualBaseDstDef}
		For a distribution $\omega\in\Dst(\finset{n})$ we define signed
		\emph{dual multinomial} distribution $\dualbase[K](\omega) \in
		\Sgn\big(\Mlt[K](\finset{n})\big)$ as:
		\[ \begin{array}{rcl}
			\dualbase[K](\omega)
			& \coloneqq &
			\displaystyle\sum_{\varphi \in \Mlt[K](\finset n)} 
			\frac{K!}{(K \shortplus n \shortminus 1)!} \cdot 
			\dualbasevector{\varphi}(\omega) \, \bigket{\varphi}.
		\end{array} \]
		
		\noindent We thus get a `signed' channel $\dualbase[K] \colon \Dst(\finset{n})
		\rightarrow \Sgn\big(\Mlt(\finset{n})\big)$.
	\end{definition}

	For instance:
	\[ \begin{array}{rcl}
		\lefteqn{\textstyle\dualbase[2]\Big(\frac{1}{2}\ket{0} + \frac{1}{6}\ket{1} + 
			\frac{1}{3}\ket{2}\Big)}
		\\[+0.4em]
		& = &
		-\frac{1}{4}\Bigket{2\ket{0}} + \frac{1}{6}\Bigket{1\ket{0} + 1\ket{1}} 
		-\frac{1}{4}\Bigket{2\ket{1}} +
                \frac{11}{6}\Bigket{1\ket{0} + 1\ket{2}} + 
		\frac{1}{6}\Bigket{1\ket{1} + 1\ket{2}} - \frac{2}{3}\Bigket{2\ket{2}}.
	\end{array} \]
	
	\noindent We have no (operational) interpretation for these distributions, but they do make sense mathematically, as in the next result.
	
	\begin{theorem}
		\label{SgnHypergeomDaggerThm}
		\begin{enumerate}
			\item There is an equality of string diagrams
                          which involves swapping dual and ordinary
                          Dirichlet distributions
			\[ \hbox{\begin{tikzpicture}
	\begin{pgfonlayer}{nodelayer}
		\node [style=medium box] (0) at (3, -1.75) {$\uniform[{\Mlt[L](\finset{n})}]$};
		\node [style=medium box] (1) at (-5.75, -1.75) {$\uniform[{\Mlt[K](\finset{n})}]$};
		\node [style=black dot] (2) at (-5.75, -0.5) {};
		\node [style=black dot] (3) at (3, -0.5) {};
		\node [style=none] (4) at (-4.25, 0.75) {};
		\node [style=none] (5) at (-7.25, 0.75) {};
		\node [style=none] (6) at (1.5, 0.75) {};
		\node [style=none] (7) at (4.5, 0.75) {};
		\node [style=small box] (8) at (-4.25, 1.5) {$\dualdirdst$};
		\node [style=none] (9) at (4.5, 4.5) {};
		\node [style=none] (10) at (1.5, 4.5) {};
		\node [style=none] (11) at (-4.25, 4.5) {};
		\node [style=none] (12) at (-7.25, 4.5) {};
		\node [style=small box] (13) at (1.5, 1.5) {$\dirdst(\one\shortplus -)$};
		\node [style=small box] (14) at (1.5, 3.25) {$\dualbase[K]$};
		\node [style=none] (15) at (-1.75, 1) {$=$};
		\node [style=small box] (16) at (-4.25, 3.25) {$\multinomial[L]$};
	\end{pgfonlayer}
	\begin{pgfonlayer}{edgelayer}
		\draw (2) to (1);
		\draw (3) to (0);
		\draw (4.center) to (8);
		\draw [in=-90, out=30] (2) to (4.center);
		\draw [in=-90, out=150] (2) to (5.center);
		\draw (5.center) to (12.center);
		\draw (10.center) to (14);
		\draw (14) to (13);
		\draw (13) to (6.center);
		\draw [in=150, out=-90] (6.center) to (3);
		\draw [in=-90, out=30] (3) to (7.center);
		\draw (7.center) to (9.center);
		\draw (11.center) to (16);
		\draw (16) to (8);
	\end{pgfonlayer}
\end{tikzpicture}
} \]
			
			\item The signed hypergeometric channel $\sgnhypergeometric[L] \colon
			\Mlt[K](\finset{n}) \rightarrow \Sgn\big(\Mlt[L](\finset{n})\big)$
			is the dagger of the composite $\dualbase[K] \klafter
			\dirdst(\one\shortplus-) \colon \Mlt[L](\finset{n}) \rightarrow
			\Sgn\big(\Mlt[K](\finset{n})\big)$ with the uniform distribution
			$\uniform[{\Mlt[L](\finset{n})}]$ as prior. In a formula:
			\[ \begin{array}{rcl}
				\sgnhypergeometric[L]
				& = &
				\Big(\dualbase[K] \klafter \dirdst(\one\shortplus-)\Big)^{\dag}_{\uniform[{\Mlt[L](\finset{n})}]}.
			\end{array} \]
		\end{enumerate}
	\end{theorem}

	\begin{proof}
		\begin{enumerate}
			\item Let $\varphi\in\Mlt[K](\finset{n})$ and
			$\psi\in\Mlt[L](\finset{n})$.
			\[ \hspace*{-2em}\begin{array}{rcl}
			  \Big(\bind{\tuple{\dualbase[K] \klafter 
							\dirdst(\one\shortplus-), \idmap}}
					{\uniform[{\Mlt[L](\finset{n})}]}\Big)(\varphi,\psi)
				& = &
				\displaystyle\frac{1}{\bibinom{n}{L}}\cdot \int_{\vec{r}}
				\frac{K!}{(K \shortplus n \shortminus 1)!} \cdot 
				\dualbasevector{\varphi}(\omega) \cdot 
				\frac{(L\shortplus n \shortminus 1)!}{\facto{\psi}} \cdot 
				\vec{r}^{\psi} \intd\vec{r}
				\\[+1.4em]
				& = &
				\displaystyle\frac{K!\cdot (n\shortminus 1)!}{(K \shortplus n \shortminus 1)!} 
				\cdot \int_{\vec{r}} \dualbasevector{\varphi}(\omega) \cdot 
				\frac{L!}{\facto{\psi}} \cdot \vec{r}^{\psi} \intd\vec{r}
				\\[+1.0em]
				& = &
				\displaystyle\frac{1}{\bibinom{n}{K}}\cdot \int_{\vec{r}}
				\dualbasevector{\varphi}(\vec{r}) \cdot 
				\multinomialvector{\psi}(\vec{r}) \intd\vec{r}
				\\[+1.2em]
				& = &
				\displaystyle\frac{1}{\bibinom{n}{K}}\cdot 
				\Big(\bind{\multinomial[L]}{\dualdirdst(\varphi)}\Big)(\psi)
				\\[+1.0em]
				& = &
				\Big(\bind{\tuple{\idmap, \multinomial[L] \klafter \dualdirdst}}
				{\uniform[{\Mlt[K](\finset{n})}]}\Big)(\varphi,\psi).
			\end{array} \]
			
			\item This is a reformulation of the previous
                          point, using that $\sgnhypergeometric[L] =
                          \multinomial[L] \klafter \dualdirdst$,
                          occurring in the above string diagram on the
                          left of the equation. \QED
		\end{enumerate}
	\end{proof}

	In Definition~\ref{DualDirichletDef} we have introduced the signed
	hypergeometric $\sgnhypergeometric[L]$ as $\multinomial[L] \klafter
	\dualdirdst$ via the dual Dirichlet distribution $\dualdirdst$. The
	above result tells that we can also obtain $\sgnhypergeometric[L]$ as
	dagger from ordinary Dirichlet $\dirdst$ (plus the dual multinomial
	distribution $\dualbasename$). This diagrammatic description of the
	dagger coincides with the one on right in Figure~\ref{DirichletFig}.

	\section{Conclusions and further work}\label{ConclusionSec}
	
	This paper covers a fascinating topic, namely negative
        probabilities.  It does not offer operational meaning, but it
        does provide a solid mathematical basis for the emergence of
        negative probabilities in classical, non-quantum probability
        theory. The techniques of categorical probability provide the
        toolbox for describing the relevant properties.
	
	There is plenty of further work. High on our list is an explicit
	formula for the dual basis vectors in the general multivariate case
	(like in the bivariate case). Also, we would like to develop a deeper
	understanding of the conjugate situation described in the string
	diagrams in Theorem~\ref{SgnHypergeomDaggerThm}.

\bibliographystyle{eptcs}
\bibliography{refs.bib}

\begin{thebibliography}{10}
\providecommand{\bibitemdeclare}[2]{}
\providecommand{\surnamestart}{}
\providecommand{\surnameend}{}
\providecommand{\urlprefix}{Available at }
\providecommand{\url}[1]{\texttt{#1}}
\providecommand{\href}[2]{\texttt{#2}}
\providecommand{\urlalt}[2]{\href{#1}{#2}}
\providecommand{\doi}[1]{doi:\urlalt{https://doi.org/#1}{#1}}
\providecommand{\eprint}[1]{arXiv:\urlalt{https://arxiv.org/abs/#1}{#1}}
\providecommand{\bibinfo}[2]{#2}

\bibitemdeclare{inproceedings}{AbramskyB14}
\bibitem{AbramskyB14}
\bibinfo{author}{S.~\surnamestart Abramsky\surnameend} \&
  \bibinfo{author}{A.~\surnamestart Brandenburger\surnameend}
  (\bibinfo{year}{2014}): \emph{\bibinfo{title}{An Operational Interpretation
  of Negative Probabilities and No-Signalling Models}}.
\newblock In \bibinfo{editor}{F.~van \surnamestart Breugel\surnameend},
  \bibinfo{editor}{E.~\surnamestart Kashefi\surnameend},
  \bibinfo{editor}{C.~\surnamestart Palamidessi\surnameend} \&
  \bibinfo{editor}{J.~\surnamestart Rutten\surnameend}, editors: {\slshape
  \bibinfo{booktitle}{Horizons of the Mind. A Tribute to Prakash Panangaden}},
  {\slshape \bibinfo{series}{Lect. Notes Comp. Sci.}} \bibinfo{volume}{8464},
  \bibinfo{publisher}{Springer, Berlin}, pp. \bibinfo{pages}{59--75},
  \doi{10.1007/978-3-319-06880-0_3}.

\bibitemdeclare{article}{blass2021negative}
\bibitem{blass2021negative}
\bibinfo{author}{A.~\surnamestart Blass\surnameend} \&
  \bibinfo{author}{Y.~\surnamestart Gurevich\surnameend}
  (\bibinfo{year}{2021}): \emph{\bibinfo{title}{Negative probabilities: what
  are they for?}}
\newblock {\slshape \bibinfo{journal}{Journal of Physics A: Mathematical and
  Theoretical}} \bibinfo{volume}{54}(\bibinfo{number}{31}), p.
  \bibinfo{pages}{315303}, \doi{10.1088/1751-8121/abef4d}.

\bibitemdeclare{article}{ChoJ19}
\bibitem{ChoJ19}
\bibinfo{author}{K.~\surnamestart Cho\surnameend} \&
  \bibinfo{author}{B.~\surnamestart Jacobs\surnameend} (\bibinfo{year}{2019}):
  \emph{\bibinfo{title}{Disintegration and {Bayesian} inversion via string
  diagrams}}.
\newblock {\slshape \bibinfo{journal}{Math. Struct. in Comp. Sci.}}
  \bibinfo{volume}{29(7)}, pp. \bibinfo{pages}{938--971},
  \doi{10.1017/s0960129518000488}.

\bibitemdeclare{inproceedings}{ClercDDG17}
\bibitem{ClercDDG17}
\bibinfo{author}{F.~\surnamestart Clerc\surnameend},
  \bibinfo{author}{F.~\surnamestart Dahlqvist\surnameend},
  \bibinfo{author}{V.~\surnamestart Danos\surnameend} \&
  \bibinfo{author}{I.~\surnamestart Garnier\surnameend} (\bibinfo{year}{2017}):
  \emph{\bibinfo{title}{Pointless learning}}.
\newblock In \bibinfo{editor}{J.~\surnamestart Esparza\surnameend} \&
  \bibinfo{editor}{A.~\surnamestart Murawski\surnameend}, editors: {\slshape
  \bibinfo{booktitle}{Foundations of Software Science and Computation
  Structures}}, {\slshape \bibinfo{series}{Lect. Notes Comp. Sci.}}
  \bibinfo{volume}{10203}, \bibinfo{publisher}{Springer, Berlin}, pp.
  \bibinfo{pages}{355--369}, \doi{10.1007/978-3-662-54458-7_21}.

\bibitemdeclare{article}{diaconis1977finite}
\bibitem{diaconis1977finite}
\bibinfo{author}{P.~\surnamestart Diaconis\surnameend} (\bibinfo{year}{1977}):
  \emph{\bibinfo{title}{Finite forms of {D}e {F}inetti's theorem on
  exchangeability}}.
\newblock {\slshape \bibinfo{journal}{Synthese}}
  \bibinfo{volume}{36}(\bibinfo{number}{2}), pp. \bibinfo{pages}{271--281},
  \doi{10.1007/BF00486116}.

\bibitemdeclare{article}{dong1993dual}
\bibitem{dong1993dual}
\bibinfo{author}{W.~\surnamestart Dong-Bing\surnameend} (\bibinfo{year}{1993}):
  \emph{\bibinfo{title}{Dual bases of a {B}ernstein polynomial basis on
  simplices}}.
\newblock {\slshape \bibinfo{journal}{Computer aided geometric design}}
  \bibinfo{volume}{10}(\bibinfo{number}{6}), pp. \bibinfo{pages}{483--489},
  \doi{10.1016/0167-8396(93)90025-X}.

\bibitemdeclare{inproceedings}{Feynman87}
\bibitem{Feynman87}
\bibinfo{author}{R.~\surnamestart Feynman\surnameend} (\bibinfo{year}{1987}):
  \emph{\bibinfo{title}{Negative Probability}}.
\newblock In \bibinfo{editor}{B.~\surnamestart Hiley\surnameend} \&
  \bibinfo{editor}{F.~\surnamestart Peat\surnameend}, editors: {\slshape
  \bibinfo{booktitle}{Quantum Implications, Essays in Honor of David Bohm}},
  \bibinfo{publisher}{Routledge and Kegan Paul, London}, pp.
  \bibinfo{pages}{235--246}, \doi{10.1063/1.2811503}.

\bibitemdeclare{article}{Fritz20}
\bibitem{Fritz20}
\bibinfo{author}{T.~\surnamestart Fritz\surnameend} (\bibinfo{year}{2020}):
  \emph{\bibinfo{title}{A synthetic approach to {Markov} kernels, conditional
  independence, and theorems on sufficient statistics}}.
\newblock {\slshape \bibinfo{journal}{Advances in Math.}}
  \bibinfo{volume}{370}, p. \bibinfo{pages}{107239},
  \doi{10.1016/J.AIM.2020.107239}.

\bibitemdeclare{book}{hoschek1993fundamentals}
\bibitem{hoschek1993fundamentals}
\bibinfo{author}{J.~\surnamestart Hoschek\surnameend} \&
  \bibinfo{author}{D.~\surnamestart Lasser\surnameend} (\bibinfo{year}{1993}):
  \emph{\bibinfo{title}{Fundamentals of computer aided geometric design}}.
\newblock \bibinfo{publisher}{AK Peters, Ltd.}

\bibitemdeclare{inproceedings}{Jacobs16a}
\bibitem{Jacobs16a}
\bibinfo{author}{B.~\surnamestart Jacobs\surnameend} (\bibinfo{year}{2016}):
  \emph{\bibinfo{title}{Affine Monads and Side-Effect-Freeness}}.
\newblock In \bibinfo{editor}{I.~\surnamestart Hasuo\surnameend}, editor:
  {\slshape \bibinfo{booktitle}{Coalgebraic Methods in Computer Science (CMCS
  2016)}}, {\slshape \bibinfo{series}{Lect. Notes Comp. Sci.}}
  \bibinfo{volume}{9608}, \bibinfo{publisher}{Springer, Berlin}, pp.
  \bibinfo{pages}{53--72}, \doi{10.1007/978-3-319-40370-0_5}.

\bibitemdeclare{article}{Jacobs18c}
\bibitem{Jacobs18c}
\bibinfo{author}{B.~\surnamestart Jacobs\surnameend} (\bibinfo{year}{2018}):
  \emph{\bibinfo{title}{From Probability Monads to Commutative Effectuses}}.
\newblock {\slshape \bibinfo{journal}{Journ. of Logical and Algebraic Methods
  in Programming}} \bibinfo{volume}{94}, pp. \bibinfo{pages}{200--237},
  \doi{10.1016/j.jlamp.2016.11.006}.

\bibitemdeclare{article}{Jacobs20a}
\bibitem{Jacobs20a}
\bibinfo{author}{B.~\surnamestart Jacobs\surnameend} (\bibinfo{year}{2020}):
  \emph{\bibinfo{title}{A Channel-Based Perspective on Conjugate Priors}}.
\newblock {\slshape \bibinfo{journal}{Math. Struct. in Comp. Sci.}}
  \bibinfo{volume}{30(1)}, pp. \bibinfo{pages}{44--61},
  \doi{10.1017/S0960129519000082}.

\bibitemdeclare{inproceedings}{Jacobs21b}
\bibitem{Jacobs21b}
\bibinfo{author}{B.~\surnamestart Jacobs\surnameend} (\bibinfo{year}{2021}):
  \emph{\bibinfo{title}{From Multisets over Distributions to Distributions over
  Multisets}}.
\newblock In: {\slshape \bibinfo{booktitle}{Logic in Computer Science}},
  \bibinfo{organization}{IEEE}, \bibinfo{publisher}{Computer Science Press},
  \doi{10.1109/lics52264.2021.9470678}.

\bibitemdeclare{article}{Jacobs22a}
\bibitem{Jacobs22a}
\bibinfo{author}{B.~\surnamestart Jacobs\surnameend} (\bibinfo{year}{2022}):
  \emph{\bibinfo{title}{Urns \& Tubes}}.
\newblock {\slshape \bibinfo{journal}{Compositionality}}
  \bibinfo{volume}{4(4)}, \doi{10.32408/compositionality-4-4}.

\bibitemdeclare{inproceedings}{JacobsS20}
\bibitem{JacobsS20}
\bibinfo{author}{B.~\surnamestart Jacobs\surnameend} \&
  \bibinfo{author}{S.~\surnamestart Staton\surnameend} (\bibinfo{year}{2020}):
  \emph{\bibinfo{title}{De {Finetti's} construction as a categorical limit}}.
\newblock In \bibinfo{editor}{D.~\surnamestart Petri\c{s}an\surnameend} \&
  \bibinfo{editor}{J.~\surnamestart Rot\surnameend}, editors: {\slshape
  \bibinfo{booktitle}{Coalgebraic Methods in Computer Science (CMCS 2020)}},
  {\slshape \bibinfo{series}{Lect. Notes Comp. Sci.}} \bibinfo{volume}{12094},
  \bibinfo{publisher}{Springer, Berlin}, pp. \bibinfo{pages}{90--111},
  \doi{10.1007/978-3-030-57201-3_6}.

\bibitemdeclare{article}{jaynes1982some}
\bibitem{jaynes1982some}
\bibinfo{author}{E.~\surnamestart Jaynes\surnameend} (\bibinfo{year}{1982}):
  \emph{\bibinfo{title}{Some applications and extensions of the {D}e {F}inetti
  representation theorem}}.
\newblock {\slshape \bibinfo{journal}{Bayesian Inference and Decision
  Techniques with Applications: Essays in Honor of Bruno de Finetti.
  North-Holland Publishers}}.

\bibitemdeclare{article}{Juttler98}
\bibitem{Juttler98}
\bibinfo{author}{B.~\surnamestart J{\"u}ttler\surnameend}
  (\bibinfo{year}{1998}): \emph{\bibinfo{title}{The dual basis functions for
  the {B}ernstein polynomials}}.
\newblock {\slshape \bibinfo{journal}{Adv. Computational Mathematics}}
  \bibinfo{volume}{8}, pp. \bibinfo{pages}{345--352},
  \doi{10.1023/A:1018912801267}.

\bibitemdeclare{article}{kerns2006definetti}
\bibitem{kerns2006definetti}
\bibinfo{author}{G.~\surnamestart Kerns\surnameend} \&
  \bibinfo{author}{G.~\surnamestart Sz{\'e}kely\surnameend}
  (\bibinfo{year}{2006}): \emph{\bibinfo{title}{De {F}inetti’s {T}heorem for
  abstract finite exchangeable sequences}}.
\newblock {\slshape \bibinfo{journal}{Journal of Theoretical Probability}}
  \bibinfo{volume}{19}(\bibinfo{number}{3}), pp. \bibinfo{pages}{589--608},
  \doi{10.1007/s10959-006-0028-z}.

\bibitemdeclare{book}{lorentz2013bernstein}
\bibitem{lorentz2013bernstein}
\bibinfo{author}{G.~\surnamestart Lorentz\surnameend} (\bibinfo{year}{2013}):
  \emph{\bibinfo{title}{Bernstein polynomials}}.
\newblock \bibinfo{publisher}{American Mathematical Soc.}

\bibitemdeclare{article}{MeissnerB11}
\bibitem{MeissnerB11}
\bibinfo{author}{G.~\surnamestart Meissner\surnameend} \&
  \bibinfo{author}{M.~\surnamestart Burgin\surnameend} (\bibinfo{year}{2011}):
  \emph{\bibinfo{title}{Negative Probabilities in Financial Modeling}}.
\newblock {\slshape \bibinfo{journal}{Wilmott Magazine}},
  \doi{10.2139/ssrn.1773077}.

\bibitemdeclare{book}{Panangaden09}
\bibitem{Panangaden09}
\bibinfo{author}{P.~\surnamestart Panangaden\surnameend}
  (\bibinfo{year}{2009}): \emph{\bibinfo{title}{Labelled {Markov} Processes}}.
\newblock \bibinfo{publisher}{Imperial College Press},
  \bibinfo{address}{London}, \doi{10.1142/p595}.

\bibitemdeclare{article}{Szekely05}
\bibitem{Szekely05}
\bibinfo{author}{G.~\surnamestart Sz\'ekely\surnameend} (\bibinfo{year}{2005}):
  \emph{\bibinfo{title}{Half of a Coin: Negative Probabilities}}.
\newblock {\slshape \bibinfo{journal}{Wilmott Magazine}}, pp.
  \bibinfo{pages}{66--68}.

\bibitemdeclare{article}{tijms2007negative}
\bibitem{tijms2007negative}
\bibinfo{author}{H.~\surnamestart Tijms\surnameend} \&
  \bibinfo{author}{K.~\surnamestart Staats\surnameend} (\bibinfo{year}{2007}):
  \emph{\bibinfo{title}{Negative probabilities at work in the {M}/{D}/1
  queue}}.
\newblock {\slshape \bibinfo{journal}{Probability in the Engineering and
  Informational Sciences}} \bibinfo{volume}{21}(\bibinfo{number}{1}), pp.
  \bibinfo{pages}{67--76}, \doi{10.1017/S0269964807070040}.

\bibitemdeclare{article}{zhao1988dual}
\bibitem{zhao1988dual}
\bibinfo{author}{K.~\surnamestart Zhao\surnameend} \&
  \bibinfo{author}{J.~\surnamestart Sun\surnameend} (\bibinfo{year}{1988}):
  \emph{\bibinfo{title}{Dual bases of multivariate {B}ernstein-{B}{\'e}zier
  polynomials}}.
\newblock {\slshape \bibinfo{journal}{Computer aided geometric design}}
  \bibinfo{volume}{5}(\bibinfo{number}{2}), pp. \bibinfo{pages}{119--125},
  \doi{10.1016/0167-8396(88)90026-X}.

\end{thebibliography}

\appendix
\section*{Appendix}
\section{Examples}\label{ExampleAppendix}

We elaborate two examples: first we illustrate how to actually compute
a dual basis, and then how to compute signed hypergeometric
distributions.

\begin{example}
	\label{BaseCoefficientEx}
	We take $n=3$ and $K=2$ so that we have
	$\bibinom{n}{K} = \frac{4!}{2!\cdot 2!} = 6$ multisets
	in $\Mlt[2](\finset{3})$, namely: $2\ket{0}$,
	$\,1\ket{0} + 1\ket{1}$, $\,2\ket{1}$, $\,1\ket{0} +
	1\ket{2}$, $\,1\ket{1} + 1\ket{2}$,
	$\,2\ket{2}$. Using this order of multisets we get a
	$6\times 6$ matrix:
	\[ \begin{array}{rccclcrcl}
		\mathsl{FS}
		& = &
		\Big(\facto{(\varphi+\psi)}\!\Big)_{\!\varphi,\psi\in\Mlt[2](\finset{3})}
		& = &
		\left(\begin{matrix}
			24 & 6 & 4 & 6 & 2 & 4
			\\
			6 & 4 & 6 & 2 & 2 & 2
			\\
			4 & 6 & 24 & 2 & 6 & 4
			\\
			6 & 2 & 2 & 4 & 2 & 6
			\\
			2 & 2 & 6 & 2 & 4 & 6
			\\
			4 & 2 & 4 & 6 & 6 & 24
		\end{matrix}\right)
		& \quad\mbox{so}\quad &
		\mathsl{FS}^{-1}
		& = &
		\displaystyle\frac{1}{60}\left(\begin{matrix}
			6 & -8 & 1 & -8 & 2 & 1
			\\
			-8 & 44 & -8 & -6 & -6 & 2
			\\
			1 & -8 & 6 & 2 & -8 & 1
			\\
			-8 & -6 & 2 & 44 & -6 & -8
			\\
			2 & -6 & -8 & -6 & 44 & -8
			\\
			1 & 2 & 1 & -8 & -8 & 6
		\end{matrix}\right)
	\end{array} \]
	
	\noindent Notice that negative values appear in this
	inverse matrix, without clear pattern.
	
	We can now compute the dual basis vectors by
	combining~\eqref{DualBasisEqn}
	and~\eqref{DualBasisCoefficientEqn}.  We elaborate the
	case of the first multiset $2\ket{0}$. For
	$(r_{0},r_{1},r_{2}) \in \simplex{3}$, that is, for
	$r_{0},r_{1},r_{2}\in [0,1]$ with $r_{0} + r_{1} +
	r_{2} = 1$ we get $\dualbasevector{2\ket{0}} \in \HP 3
	2$ determined as:
	\[ \begin{array}{rcl}
		\lefteqn{\dualbasevector{2\ket{0}}(r_{0}, r_{1}, r_{2})}
		\\[+0.2em]
		& \smash{\stackrel{\eqref{DualBasisEqn}}{=}} &
		\displaystyle\sum_{\varphi\in\Mlt[2](\finset{3})} 
		c_{\varphi,2\ket{0}} \cdot (r_{0}, r_{1}, r_{2})^{\varphi}
		
		\hspace*{\arraycolsep}\smash{\stackrel{\eqref{DualBasisCoefficientEqn}}{=}}\hspace*{\arraycolsep}
		\displaystyle\sum_{\varphi\in\Mlt[2](\finset{3})} \textstyle
		\frac{6!}{\coefm{2\ket{0}}}\cdot 
		\Big(\mathsl{FS}^{-1}\Big)_{\!\varphi,2\ket{0}}
		\cdot r_{0}^{\varphi(0)}\cdot r_{1}^{\varphi(1)}\cdot r_{2}^{\varphi(2)}
		\\[+1.4em]
		& = &
		12\Big(6r_{0}^{2} -8r_{0}^{1}r_{1}^{1} + 1r_{1}^{2} 
		-8r_{0}^{1}r_{2}^{1} + 2r_{1}^{1}r_{2}^{1} + 1r_{2}^{2}\Big)
		\hspace*{\arraycolsep}=\hspace*{\arraycolsep}
		72r_{0}^{2} - 96r_{0}r_{1} + 12r_{1}^{2} - 96r_{0}r_{2} + 
		24r_{1}r_{2} + 12r_{2}^{2}.
	\end{array} \]
	
	\noindent Similarly, for the other multisets in $\Mlt[2](\finset{3})$,
	\[ \begin{array}{rcl}
		\dualbasevector{1\ket{0}+1\ket{1}}(r_{0}, r_{1}, r_{2})
		& = &
		-48r_{0}^{2} + 264r_{0}r_{1} - 48r_{1}^{2} - 36r_{0}r_{2} -
		36r_{1}r_{2} + 12r_{2}^{2}
		\\[+0.3em]
		\dualbasevector{2\ket{1}}(r_{0}, r_{1}, r_{2})
		& = &
		12r_{0}^{2} - 96r_{0}r_{1} + 72r_{1}^{2} + 24r_{0}r_{2} -
		96r_{1}r_{2} + 12r_{2}^{2}
		\\[+0.3em]
		\dualbasevector{1\ket{0}+1\ket{2}}(r_{0}, r_{1}, r_{2})
		& = &
		-48r_{0}^{2} - 36r_{0}r_{1} + 12r_{1}^{2} + 264r_{0}r_{2} -
		36r_{1}r_{2} - 48r_{2}^{2}
		\\[+0.3em]
		\dualbasevector{1\ket{1}+1\ket{2}}(r_{0}, r_{1}, r_{2})
		& = &
		12r_{0}^{2} - 36r_{0}r_{1} - 48r_{1}^{2} - 36r_{0}r_{2} +
		264r_{1}r_{2} - 48r_{2}^{2}
		\\[+0.3em]
		\dualbasevector{2\ket{2}}(r_{0}, r_{1}, r_{2})
		& = &
		12r_{0}^{2} + 24r_{0}r_{1} + 12r_{1}^{2} - 96r_{0}r_{2} -
		96r_{1}r_{2} + 72r_{2}^{2}.
	\end{array} \]
	
	\noindent Then indeed,
	$\inprod{\multinomialvector{\varphi}}{\dualbasevector{\psi}} =
	\delta_{\varphi,\psi}$ for all $\varphi,\psi\in\Mlt[2](\finset{3})$.
	
	We check the claim of Proposition~\ref{DualBaseSumProp}, that
        the sum of the dual base vectors is a particular
        constant. Let's abbreviate the (pointwise) sum of the above
        dual basis functions as: $\dualbasevector{} \coloneqq
        \sum_{\varphi\in\Mlt[2](\finset{3})}
        \dualbasevector{\varphi}$. Then, using the above descriptions,
        for $(r_{0}, r_{1}, r_{2}) \in \Delta^{3}$,
	\[ \begin{array}{rcl}
		\dualbasevector{}(r_{0}, r_{1}, r_{2})
		& = &
		12r_{0}^{2} + 24r_{0}r_{1} + 12r_{1}^{2} + 24r_{0}r_{2} +
		24r_{1}r_{2} + 12r_{2}^{2}
		\\[+0.3em]
		& = &
		12\Big((r_{0} + r_{1})^{2} + 2(r_{0}+r_{1})r_{2} + r_{2}^{2}\Big)
		\hspace*{\arraycolsep}=\hspace*{\arraycolsep}
		12\big(r_{0} + r_{1} + r_{2}\big)^{2}
		\hspace*{\arraycolsep}=\hspace*{\arraycolsep}
		12.
	\end{array} \]
	
	\noindent Hence the sum of the dual basis functions is indeed
        a constant and equals $\frac{(K+n-1)!}{K!} = \frac{4!}{2!} =
        12$.
\end{example}

\begin{example}
	\label{SgnHypergeomEx}
	Let's take as urn $\upsilon = 1\ket{0} + 1\ket{1} + 1\ket{2}$ with one
	ball of each color in $\finset{3} = \{0,1,2\}$. By drawing two balls
	we remain within the world of ordinary hypergeometric distributions,
	by Theorem~\ref{SgnDirichletHypgeomThm} and get the following
	distribution over draws.
	\[ \begin{array}{rcl}
		\sgnhypergeometric[2](\upsilon) =
		\hypergeometric[2](\upsilon)
		& = &
		\frac{1}{3}\bigket{1\ket{0} + 1\ket{1}} +
		\frac{1}{3}\bigket{1\ket{0} + 1\ket{2}} +
		\frac{1}{3}\bigket{1\ket{1} + 1\ket{2}}.
	\end{array} \]
	
	\noindent We can also draw three balls; in that case, the outcome is certain to be the whole urn:
	\[ \begin{array}{rcccccl}
		\sgnhypergeometric[3](\upsilon)
		&  = &
		\hypergeometric[3](\upsilon)
		& = &
		1\bigket{1\ket{0} + 1\ket{1} + 2\ket{1}}
		& = &
		1\bigket{\upsilon}.
	\end{array} \]
	
	\noindent Drawing four balls, more than in the urn, is only possible
	with the signed hypergeometric. It leads to negative probabilities in:
	\[ \begin{array}{rcl}
		\sgnhypergeometric[4](\upsilon)
		& = &
		\frac{7}{126}\bigket{4\ket{0}} -
		\frac{14}{126}\bigket{3\ket{0} + 1\ket{1}} -
		\frac{21}{126}\bigket{2\ket{0} + 2\ket{1}} -
		\frac{14}{126}\bigket{1\ket{0} + 3\ket{1}} +
		\frac{7}{126}\bigket{4\ket{1}} 
		\\[+0.2em]
		& & \; - \,
		\frac{14}{126}\bigket{3\ket{0} + 1\ket{2}} +
		\frac{84}{126}\bigket{2\ket{0} + 1\ket{1} + 1\ket{2}} +
		\frac{84}{126}\bigket{1\ket{0} + 2\ket{1} + 1\ket{2}} 
		\\[+0.2em]
		& & \; - \,
		\frac{14}{126}\bigket{3\ket{1} + 1\ket{2}} -
		\frac{21}{126}\bigket{2\ket{0} + 2\ket{2}} +
		\frac{84}{126}\bigket{1\ket{0} + 1\ket{1} + 2\ket{2}}
		\\[+0.2em]
		& & \;\; - \,
		\frac{21}{126}\bigket{2\ket{1} + 2\ket{2}} -
		\frac{14}{126}\bigket{1\ket{0} + 3\ket{2}} -
		\frac{14}{126}\bigket{1\ket{1} + 3\ket{2}} +
		\frac{7}{126}\bigket{4\ket{2}}.
	\end{array} \]
	
	
	\noindent There is no apparent `logic' in these probabilities, for
	instance in terms of draw probabilities. In particular, an intuitive
	explanation is missing for why certain probabilities are negative.
	But, as explained above, these distributions are constructed in a
	systematic manner, via dual bases.
	
	Clarification: the outcomes of the signed hypergeometric given
        above are obtained via elementary Python scripts that perform
        matrix inversion --- as in~\eqref{DualBasisCoefficientEqn} ---
        to obtain representations of dual basis vectors and to perform
        integration over simplices. Our Python scripts produce real
        numbers as probabilities, but they are so close to the above
        fractions in $\sgnhypergeometric[4](\upsilon)$ that we write
        these fractions instead, for the sake of readability.
\end{example}

\section{Missing proofs}\label{ProofAppendix}

In the body of the article we left out the proofs of several basis
properties of ordinary draws.

\begin{proof} (of Theorem~\ref{DirichletThm})
	\begin{enumerate}
		\item Via~\eqref{ContinuousBindEqn}, for arbitrary
		$\varphi\in\Mlt[K](\finset{n})$,
		\[ \begin{array}{rcl}
			\big(\bind{\multinomial[K]}{\dirdst(\upsilon)}\big)(\varphi)			& = &
			\displaystyle\int_{\vec{r}\in\simplex{n}} \textstyle
			\multinomial[K](\vec{r})(\varphi) \cdot
			\dirden(\upsilon)(\vec{r}) \intd\vec{r}
			\\[+1em]
			& = &
			\displaystyle\int_{\vec{r}\in\simplex{n}}
			\coefm{\varphi} \cdot \prod_{i\in\finset{n}} r_{i}^{\varphi(i)} \cdot
			\frac{(L\shortminus 1)!}{\facto{(\upsilon\shortminus \one)}} \cdot
			\prod_{i\in\finset{n}} \, r_{i}^{\upsilon(i)-1} \intd\vec{r}
			\\[+1.2em]
			& = &
			\displaystyle\int_{\vec{r}\in\simplex{n}}
			\frac{K!\cdot (L\shortminus 1)!}
			{\facto{\varphi} \cdot \facto{(\upsilon\shortminus\one)}} \cdot
			\prod_{i\in\finset{n}} \, r_{i}^{\upsilon(i)+\varphi(i)-1} \intd\vec{r}
			\\[+1.2em]
			& = &
			\displaystyle\frac{\big(\binom{\upsilon}{\varphi}\big)}
			{\big(\binom{L}{K}\big)} \cdot \int_{\vec{r}\in\simplex{n}}
			\frac{(L\shortplus K\shortminus 1)!}
			{\facto{(\upsilon\shortplus \varphi\shortminus \one)}} \cdot
			\prod_{i\in\finset{n}} \, r_{i}^{(\upsilon+\varphi)(i)-1} \intd\vec{r}
			\\[+1.2em]
			& \smash{\stackrel{\eqref{DirichletNormalisationEqn}}{=}} &
			\polya[K](\upsilon)(\varphi).
		\end{array} \]
		
		\item This can be computed in a similar manner, but the details are
		beyond the scope of this paper.
		
		\item We have already seen that the set $\Mlt[K](\finset{n})$ has
		$\bibinom{n}{K}$ elements, so that the uniform distribution
		$\uniform[{\Mlt[K](\finset{n})}]$ on this set is
		$\sum_{\varphi\in\Mlt[K](\finset{n})}
		\frac{1}{\bibinom{n}{K}}\bigket{\varphi}$. This uniform distribution
		equals $\polya[K](\one)$, where $\one = \sum_{i\in\finset{n}}
		1\ket{i}$. Hence we can reason diagrammatically, as on the right in
		Figure~\ref{DirichletFig}. The uniform distribution
		$\uniform[\Dst(\finset{n})]$ on the left in this chain of equations
		is $\dirdst(\one)$ on $\simplex{n}$, with density $\vec{r} \mapsto
		(n-1)!$, since $\int_{\vec{r}\in\simplex{n}} 1\intd\vec{r} =
		\frac{1}{(n-1)!}$ by~\eqref{DirichletNormalisationEqn}.
		
		\item Consider an urn $\upsilon$ of size $L$ and a number
		$j\in\finset{n}$. We write $\upsilon_{j} \coloneqq \upsilon +
		1\ket{j}$ of size $L+1$. Then, using~\eqref{ContinuousBindEqn},
		\[ \begin{array}[b]{rcl}
		  \big(\bind{\sample}{\dirdst(\upsilon)}\big)(j)
			& = &
			\displaystyle\int_{\vec{r}\in\simplex{n}} r_{j} \cdot 
			\frac{(L\shortminus 1)!}{\facto{(\upsilon\shortminus\one)}}\cdot 
			\prod_{i\in\finset{n}} r_{i}^{\upsilon(i)-1} \intd \vec{r}
			\\[+1.2em]
			& = &
			\displaystyle \frac{\upsilon(j)}{L} \cdot \int_{\vec{r}\in\simplex{n}} 
			\frac{L!}{\facto{(\upsilon_{j}\shortminus\one)}}\cdot 
			\prod_{i\in\finset{n}} r_{i}^{\upsilon_{j}(i)-1} \intd \vec{r}
			\\[+0.8em]
			& \smash{\stackrel{\eqref{DirichletNormalisationEqn}}{=}} &
			\displaystyle \frac{\upsilon(j)}{L} \cdot 1
			\\[+0.5em]
			& = &
			\flrn(\upsilon).
		\end{array} \eqno{\blacktriangleleft} \]
	\end{enumerate}
\end{proof}

\section{The bivariate case}\label{BivariateAppendix}
	
	\begin{figure*}[th]
		\label{BivariateSignedHypergeometricFig}
	\end{figure*}
	
	In the previous sections we have elaborated the multivariate case,
	involving multiple variables. The bivariate case, for $n=2$, can be
	done slightly differently, using the isomorphisms $\Dst(\finset{2})
	\cong [0,1]$ and $\Mlt[K](\finset{2}) \cong \{0,1,\ldots,K\} \cong
	\finset{K\shortplus 1}$. The binomial channel $\binomial[K] \colon
	[0,1] \chanto \{0,\ldots,K\}$ is thus related via the general
	multinomial one via the following square.
	\[ \xymatrix@R-0.8pc{
		[0,1]\ar[rr]|-{\circ}^-{\binomial[K]}\ar[d]^{\cong} & & \{0,\ldots,K\}
		\\
		\Dst(\finset{2})\ar[rr]|-{\circ}^-{\multinomial[K]} & & 
		\Mlt[K](\finset{2})\ar[u]^-{\cong}
	} \]
	
	\noindent Explicitly, for $r\in [0,1]$, we have
	\[ \begin{array}{rcl}
		\binomial[K](r)
		& = &
		\displaystyle\sum_{0\leq i\leq K} \binom{K}{i}\cdot r^{i} \cdot (1-r)^{K-i}
		\,\bigket{i}.
	\end{array} \]
	
	\noindent The polynomials involved are known as \emph{Bernstein polynomials}, namely
	\[ \begin{array}{rcl}
		\binomialvector{i}(r)
		\hspace*{\arraycolsep}\coloneqq\hspace*{\arraycolsep}
		\displaystyle \binom{K}{i}\cdot r^{i} \cdot (1\shortminus r)^{K-i}
		& = &
		\displaystyle\sum_{0\leq j\leq K-i} \binom{K}{i}\cdot \binom{K\shortminus i}{j}
		\cdot r^{i} \cdot 1^{j} \cdot (-r)^{K-i-j}
		\\
		& = &
		\displaystyle\sum_{0\leq j\leq K-i} \binom{K}{i}\cdot \binom{K\shortminus i}{j}
		\cdot (-1)^{K-i-j} \cdot r^{K-j}.
	\end{array} \]
	
	\noindent These polynomials are widely studied in Computer Graphics
	and Computer Aided Geometric Design~\cite{hoschek1993fundamentals},
	but also in areas such as approximation
	theory~\cite{lorentz2013bernstein} and probability. They appear not
	only as probability mass functions for the binomial distributions (as
	described above), but also as density function of the (continuous)
	Beta distributions, rescaled by a normalisation factor.
	
	The Hilbert space $\HP{2}{K}$ can be identified with the space of univariate polynomials, as functions $[0,1] \rightarrow \R$, spanned by
	the monomial basis $(r^i : 0 \leq i \leq K)$. The $(n\shortplus
	1)\times(n\shortplus 1)$ matrix $B$ that represents the above
	polynomials $\binomialvector{i}$ in this basis has entries
	\[ \begin{array}{rcl}
		B_{i,K-j}
		& = &
		\displaystyle \binom{K}{i}\cdot \binom{K\shortminus i}{j} \cdot (-1)^{K-i-j}.
	\end{array} \]
	
	\noindent Below we plot several Bernstein polynomials (on the left)
	and their dual bases (on the right).
	\[ \includegraphics[scale=0.3]{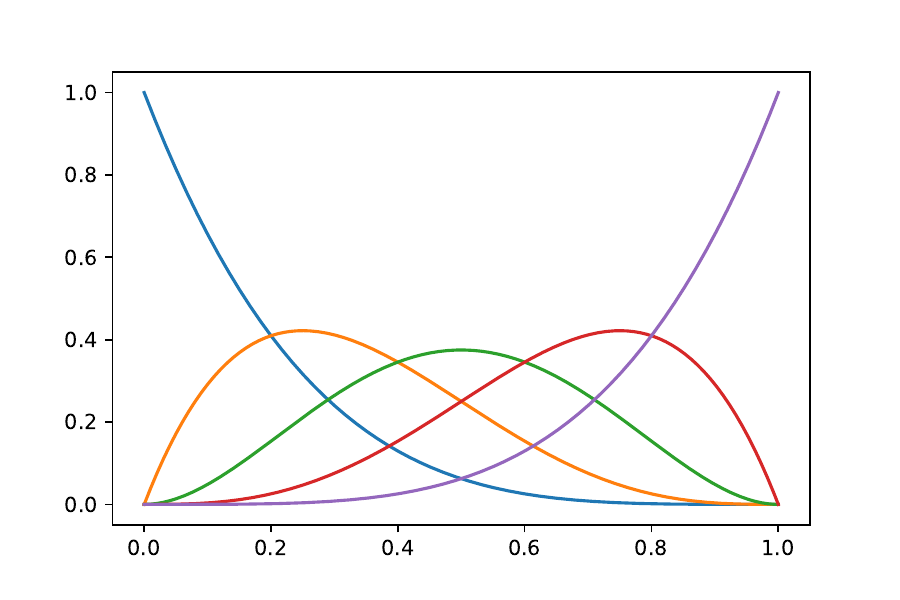} 
        \hspace*{6em}
	\includegraphics[scale=0.3]{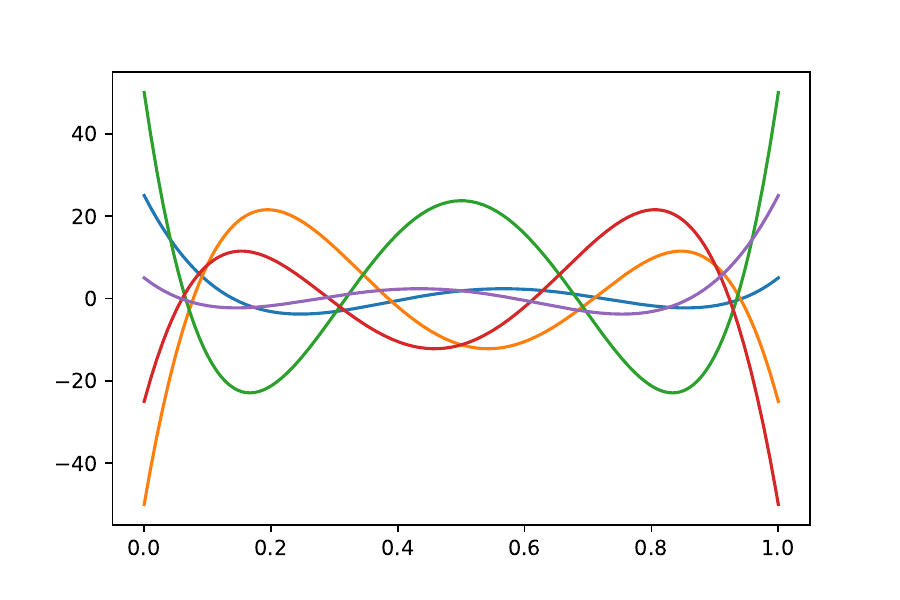}  \]
	
	We write $S$ for the matrix with inner products of the base
	vectors, so:
	\[ \begin{array}{rcccccl}
		S_{i,j}
		& \coloneqq &
		\inprod{r^{i}}{r^{j}}
		& = &
		\displaystyle\int_{r\in [0,1]} r^{i}\cdot r^{j} \intd r
		& = &
		\frac{1}{i+j+1}.
	\end{array} \]
	
	\noindent The dual basis of the binomial polynomials
	$\binomialvector{i}$ are then given by the matrix inverse $(B^{T}\cdot
	S)^{-1}$. It has been studied in a Computer Graphics context
	\cite{dong1993dual} \cite{zhao1988dual}, and formulas for computing
	the dual basis are known, see~\cite{Juttler98} for an overview.
	
	This explicit formulation of the dual basis allows us to describe the
	\emph{bivariate signed hypergeometric} channel
	$\bisgnhypergeometric[L,K] \colon \{0,\ldots,L\} \rightarrow
	\Sgn\big(\{0,\ldots,K\})$, in the following commuting diagram.
	\[ \xymatrix@R-0.8pc{
		\{0,\ldots,L\}\ar[d]_{\cong}\ar[rr]|-{\circ}^-{\bisgnhypergeometric[L,K]} 
		& & \{0,\ldots,K\}
		\\
		\Mlt[L](\finset{2})\ar[rr]|-{\circ}^-{\sgnhypergeometric[K]} & & 
		\Mlt[K](\finset{2})\ar[u]_-{\cong}
	} \]
	
	\noindent We include the parameter $L$ in writing
	$\bisgnhypergeometric[L,K]$ since it cannot be derived from an input
	$j\in\{0,\ldots,L\}$. In contrast, writing this parameter explicitly
	is not needed in the multivariate case, since the size of the urn can
	be computed from the urn itself.
	
	The explicit formula for this bivariate signed hypergeometric
        $\bisgnhypergeometric[L,K]$ is:
\begin{equation}
\label{BivariateEqn}
\begin{array}{rcl}
  \lefteqn{\bisgnhypergeometric[L,K](j)}
  \\
  & \coloneqq &
  \displaystyle \sum_{0\leq i\leq K} 
  \frac{\binom{K}{i}}{(K\shortplus L\shortplus 1) \cdot \binom{L}{j}} 
  \left(\sum_{0\leq \ell \leq L}
  \frac{(-1)^{j+\ell}}{\binom{K+L}{i+\ell}} 
  \sum_{0 \leq k\leq \min(j,\ell)} (2k\shortplus 1) 
  \binom{L\shortplus k\shortplus 1}{L\shortminus j}
  \binom{L \shortminus k}{L\shortminus j} 
  \binom{L\shortplus k\shortplus 1}{L\shortminus \ell} 
  \binom{L\shortminus k}{L\shortminus \ell}
  \right)\bigket{i}
\end{array}
\end{equation}

\noindent This $\bisgnhypergeometric[L,K]$ is defined on $0\leq j\leq
L$, corresponding to urn $j\ket{0} + (L-j)\ket{1}$. The complicated
character of this formula is not helpful for an operational
interpretation in terms of draw probabilities. But it does allow us to
compute (truely) exact distributions.
	
	\begin{example}
		Let's take an urn size $L = 3$ with $2$ balls of colour $0$.  Thus, in
		a multivariate scenario we would write this as urn $\upsilon =
		2\ket{0} + 1\ket{1}$. We first look at a draw of size $K=4$. Using the
		formula from Figure~\ref{BivariateSignedHypergeometricFig} we get a
		bivariate signed hypergeometric distribution of the form:
		\[ \begin{array}{rcl}
			\bisgnhypergeometric[3,4](2)
			& = &
			\frac{17}{210}\bigket{0} -\frac{34}{105}\bigket{1} + 
			\frac{17}{35}\bigket{2} + \frac{106}{105}\bigket{3} - 
			\frac{53}{210}\bigket{4}.
		\end{array} \]
		
		\noindent The number $i$ in $\ket{i}$ refers to the number of balls
		of colour $0$ that are drawn (out of $K$ in total), with corresponding
		(positive or negative) probability.
		
		Similarly, for $K=5$ we have:
		\[ \begin{array}{rcl}
		  \bisgnhypergeometric[3,5](2)
			& = &
			\frac{1}{6}\bigket{0} - \frac{3}{7}\bigket{1} + \frac{1}{21}\bigket{2} + 
			\frac{16}{21}\bigket{3} + \frac{37}{42}\bigket{4} - 
			\frac{3}{7}\bigket{5}.
		\end{array} \]
	\end{example}

\end{document}